\documentclass[oneside,english]{amsart}
\usepackage[T1]{fontenc}
\usepackage{amsthm}
\usepackage{amssymb}
\usepackage{esint}
\usepackage[all]{xy}
\usepackage{lmodern}
\usepackage{tikz}

\makeatletter

\numberwithin{equation}{section}
\numberwithin{figure}{section}
\theoremstyle{plain}
\newtheorem{thm}{\protect\theoremname}[section]
  \theoremstyle{plain}
  \newtheorem{prop}[thm]{\protect\propositionname}
  \theoremstyle{definition}
  \newtheorem{defn}[thm]{\protect\definitionname}
  \theoremstyle{remark}
  \newtheorem{rem}[thm]{\protect\remarkname}
  \theoremstyle{plain}
  \newtheorem{lem}[thm]{\protect\lemmaname}
  \theoremstyle{definition}
  \newtheorem{example}[thm]{\protect\examplename}
  \newtheorem{cor}[thm]{\protect\corollaryname}
  \theoremstyle{plain}

\makeatother

\usepackage{babel}
  \providecommand{\definitionname}{Definition}
  \providecommand{\examplename}{Example}
  \providecommand{\lemmaname}{Lemma}
  \providecommand{\propositionname}{Proposition}
  \providecommand{\remarkname}{Remark}
  \providecommand{\corollaryname}{Corollary}
\providecommand{\theoremname}{Theorem}

\begin{document}

\title{Realization spaces of algebraic structures on cochains}

\address{University of Copenhagen, Department of Mathematical Sciences,
Universitetsparken 5, 2100 København}

\email{yalinprop@gmail.com}

\author{Sinan Yalin}
\begin{abstract}
Given an algebraic structure on the cohomology of a cochain complex, we define its realization space as a Kan complex whose vertices are the structures up to homotopy realizing this structure at the cohomology level. Our algebraic structures are parameterized by props and thus include various kinds of bialgebras. We give a general formula to compute subsets of equivalence classes of realizations as quotients of automorphism groups, and determine the higher homotopy groups via the cohomology of deformation complexes. As a motivating example, we compute subsets of equivalences classes of realizations of Poincaré duality for several examples of manifolds.
\end{abstract}
\maketitle

\tableofcontents{}

\section*{Introduction}

\subsection*{Motivations}

Realization problems of algebraic structures appear naturally in various contexts related to topology and geometry.
It consists in determining a space whose cohomology or homotopy groups are isomorphic to a given algebra,
a cochain complex equipped with an algebraic structure up to homotopy whose cohomology is isomorphic to a given algebra,
or constructing an algebraic structure up to homotopy on a given cochain complex which induces the desired structure at the cohomology level.
In practice, one defines a moduli space of such objects (spaces, homotopy algebras or homotopy structures)
and sets up an obstruction theory to determine whether this moduli space is non empty (existence problem) and connected (uniqueness problem).
This obstruction theoretic approach to connectedness of such moduli spaces goes back to the pioneering work of Halperin and Stasheff \cite{HS}.
Such methods were successfully applied in \cite{BDG} to determine the topological spaces realizing a given $\Pi$-algebra (a graded group equipped
with the algebraic operations existing on the homotopy groups of a pointed connected topological space).
A similar approach was thoroughly developed  by Goerss and Hopkins in \cite{GH} to build obstruction theories for multiplicative structures on ring spectra,
proving in particular the existence and uniqueness of $E_{\infty}$-structures on Lubin-Tate spectra (a non trivial improvement of the Hopkins-Miller theorem).

Realization problems also occur in the differential graded setting,
and can involve algebraic structures not parameterized by operads but more general props.
Moreover, to fully understand the realization problem, one has to determine the
homotopy type of the corresponding moduli space.
In particular, counting the connected components (the equivalence classes of realizations) when the realization
is not unique is a very hard problem. Let us give two motivating examples.

The operations defining Poincaré duality on a closed connected oriented manifold organize
into a Frobenius algebra structure on its singular cohomology (when the coefficients are taken in a field), or equivalently a (unitary and counitary) Frobenius bialgebra.
As a consequence of \cite{LS}, we know that, when working over $\mathbb{Q}$, Poincaré duality can be realized at the cochain level
in a Frobenius bialgebra up to homotopy (a homotopy Frobenius bialgebra) which extends the $E_{\infty}$-structure generated by the (higher) cup products.
Briefly, this comes from the following argument line. Singular cochains with the higher cup products on a closed connected manifold $M$ form an $E_{\infty}$-algebra quasi-isomorphic to the Sullivan model of $M$ in rational homotopy theory (see for instance \cite{FOT}). This Sullivan model is a commutative dg algebra whose cohomology forms a Frobenius bialgebra, hence by \cite{LS} is quasi-isomorphic, as a commutative algebra, to a certain Frobenius bialgebra. This Frobenius bialgebra structure can then be transferred along the two quasi-isomorphisms above into a homotopy Frobenius bialgebra structure on singular cochains extending the higher cup products. Let us note that when working over $\mathbb{R}$, one can alternately use the de Rham complex of $M$ instead of its Sullivan model.
Such realizations are not well known at present, and even an explicit notion of homotopy Frobenius bialgebra was
obtained only recently in \cite{CMW}.
These higher operations are expected to contain geometric information about the underlying manifold.

String topology, introduced by Chas and Sullivan in \cite{CS1}, study the homology of loop spaces on manifolds. Due to the action of the circle on loops,
one considers either the non equivariant homology or the $S^1$-equivariant homology (called the string homology).
These graded spaces come with a rich algebraic structure of geometrical nature. In the non equivariant case, it forms a Batalin-Vilkovisky algebra,
and in the equivariant one it forms an involutive Lie bialgebra \cite{CS2}. This involutive Lie bialgebra
is expected to be induced by finer operations at the chain level, organized into a homotopy involutive Lie bialgebra.
Let us note that, even though we work in the cochain setting in this paper to apply our results to the study of Poincaré duality, our results still hold true when working with chains instead of cochains, and thus can be applied to the string topology example as well.

\subsection*{Results in the prop setting}

Let $P$ be a prop and $X$ be a cochain complex such that $H^*X$ forms a $P$-algebra.
A $P$-algebra up to homotopy, or homotopy $P$-algebra, is a an algebra over a cofibrant resolution $P_{\infty}$ of $P$.
In a $P_{\infty}$-algebra structure, the relations defining the $P$-algebra structure
are relaxed and are only satisfied up to a set of coherent (higher) homotopies.
According to \cite{Yal3}, such a notion is homotopy coherent with respect to the choice of a cofibrant resolution $P_{\infty}$.
We consider two sorts of realization spaces.
The first one, an analogue in the dg setting of those studied in \cite{GH},
is the simplicial nerve $\mathcal{N}\mathcal{R}^{P_{\infty}}$ of a category $\mathcal{R}^{P_{\infty}}$ whose objects are the $P_{\infty}$-algebras $(X,\phi)$ with a cohomology isomorphic
to $H^*X$ as a $P$-algebra, and morphisms are quasi-isomorphisms of $P_{\infty}$-algebras.
The second one is a Kan complex $Real_{P_{\infty}}(H^*X)$
whose vertices are the prop morphisms $P_{\infty}\rightarrow End_X$ inducing the $P$-algebra structure $P\rightarrow End_{H^*X}$
in cohomology. It is a simplicial subset of the simplicial mapping space $P_{\infty}\{X\}$ of maps $P_{\infty}\rightarrow End_X$ (the moduli space of $P_{\infty}$-algebra
structures on $X$). We decompose $Real_{P_{\infty}}(H^*X)$ into more manageable pieces called ``local realization spaces''
\[
Real_{P_{\infty}}(H^*X) = \coprod_{[X,\phi]\in\pi_0\mathcal{N}\mathcal{R}^{P_{\infty}}} P_{\infty}\{X\}_{[\phi]}.
\]
For a given $\phi:P_{\infty}\rightarrow End_X$, the local realization space $P_{\infty}\{X\}_{[\phi]}$ is a Kan complex whose vertices are the $P_{\infty}$-algebra
structures $\varphi:P_{\infty}\rightarrow End_X$ such that the $P_{\infty}$-algebras $(X,\varphi)$ and $(X,\phi)$ are related by
a zigzag of quasi-isomorphisms.
The main result of our paper reads:
\begin{thm}
Let $P$ be a prop with trivial differential and $P_{\infty}$ be a cofibrant resolution of $P$.
Let $X$ be a cochain complex such that $H^*X$ forms a $P$-algebra.

(1) There is an injection
\[
\coprod_{[\phi]\in\pi_0\mathcal{N}\mathcal{R}^{P_{\infty}}}
Aut_{\mathbb{K}}(H^*X)/Aut_{Ho(Ch_{\mathbb{K}}^{P_{\infty}})}(H^*X,\phi)\hookrightarrow \pi_0Real_{P_{\infty}}(H^*X)
\]
where $Aut_{\mathbb{K}}(H^*X)$ stands for the group of automorphisms of $H^*X$ as a dg $\mathbb{K}$-module and
$Aut_{Ho(Ch_{\mathbb{K}}^{P_{\infty}})}(H^*X,\phi)$ stands for the group of automorphisms of $(H^*X,\phi)$ in the homotopy category of dg $P_{\infty}$-algebras.

(2) For $n\geq 1$, the $n^{th}$ homotopy groups of the realization space $Real_{P_{\infty}}(H^*X)$ are given by
\[
\pi_n(Real_{P_{\infty}}(H^*X),\phi) \cong H^{1-n}Der_{\phi}(P_{\infty},End_{H^*X}).
\]

(3) For $n\geq 2$, we have
\begin{eqnarray*}
\pi_n(Real_{P_{\infty}}(H^*X),\phi) & \cong & \pi_n(\mathcal{N}\mathcal{R}^{P_{\infty}},(H^*X,\phi))\\
 & \cong & \pi_{n-1}(L^HwCh_{\mathbb{K}}^{P_{\infty}}((H^*X,\phi),(H^*X,\phi)),id)
\end{eqnarray*}
where $L^H$ is the hammock localization functor constructed in \cite{DK2}.
\end{thm}
Note that $L^HwCh_{\mathbb{K}}^{P_{\infty}}(-,-)$ is the simplicial monoid of homotopy automorphisms.
When $P$ is an operad, then the $P_{\infty}$-algebras form a model category and we recover
the usual simplicial monoid of self weak equivalences $haut(-)$.
To put it into words, one can compute connected components of the moduli space of realizations as quotients of automorphisms groups,
its higher homotopy groups at a given base point are isomorphic to the cohomology of the deformation complex of this base point,
and for $n\geq 2$ we recover the homotopy groups of the "Goerss-Hopkins" realization space and the homotopy groups of homotopy automorphisms.
The proof of Theorem 0.1 is as follows. We consider a restriction of the homotopy fiber sequence of \cite[Theorem 0.1]{Yal2} to local realization spaces,
we explain how we can replace it with an actual Kan fibration up to weak equivalences obtained from a certain simplicial Borel construction, and we use long exact sequences arguments.
The computation of higher homotopy groups follows from a correspondence between higher homotopy groups of mapping spaces and cohomology groups of deformation complexes.
The proof of this intermediate result relies on Lie theory for $L_{\infty}$-algebras and is fully written in \cite{Yal4}.

We then study extensions of algebraic structures. Precisely, if we have a prop morphism $O\rightarrow P$, it implies that every $P$-algebra is in particular a $O$-algebra. The realization problem is the following. The prop morphism $O\rightarrow P$ induces a prop morphism $i:O_{\infty}\rightarrow P_{\infty}$. Let $X$ be a cochain complex such that $H^*X$ is a $P$-algebra,
and suppose that $X$ possesses a $O_{\infty}$-algebra structure given by a morphism $\psi:O_{\infty}\rightarrow End_X$. The extensions of this $O_{\infty}$-algebra structure into $P_{\infty}$-algebra structures are parametrized by the homotopy fiber
\[
P_{\infty}\{X,\psi\} \rightarrow P_{\infty}\{X\}\stackrel{i^*}{\rightarrow}O_{\infty}\{X\},
\]
where $i^*$ is the map induced by precomposing with $i$ and the fiber is taken over the base point $\psi$.
Our first result in the relative setting is a generalization of the main theorem of \cite{Yal2}:
\begin{thm}
Let $(X,\psi)$ be a dg $O_{\infty}$-algebra.
Then the commutative square
\[
\xymatrix{P_{\infty}\{X,\psi\}\ar[d]\ar[r] & \mathcal{N}wCh_{\mathbb{K}}^{P_{\infty}}\ar[d]^{\mathcal{N}i^*}\\
\{X,\psi\}\ar[r] & \mathcal{N}wCh_{\mathbb{K}}^{O_{\infty}}
}
\]
is a homotopy pullback of simplicial sets.
\end{thm}
We then use a simplicial Borel construction to get results similar to those of the non relative context.
In particular, for local realization spaces there is a injection of connected components
\[
Aut_{Ho(Ch_{\mathbb{K}}^{O_{\infty}})}(X,\psi)/
Aut_{Ho(Ch_{\mathbb{K}}^{P_{\infty}})}(X,\phi)\hookrightarrow \pi_0P_{\infty}\{X,\psi\}_{[\phi]}.
\]

\subsection*{Results in the operad setting}

When $P_{\infty}$ is an operad, it turns out that $Real_{P_{\infty}}(H^*X)$ has a nice interpretation as the moduli space of $P_{\infty}$-algebras with underlying complex
$X$ realizing the $P$-algebra structure on $H^*X$, such that $\pi_0Real_{P_{\infty}}(H^*X)$ is the set of $\infty$-isotopy classes of such algebras ($\infty$-morphisms of $P_{\infty}$-algebras
which reduce to the identity on $X$).
Using the obstruction theory for algebras over operads developed in \cite{Hof} in terms of gamma cohomology, we get:
\begin{thm}
Let $P$ be a $\Sigma$-cofibrant graded operad with a trivial differential. Let $H^*X$ be a $P$-algebra
such that $H\Gamma^0_P(H^*X,H^*X)=H\Gamma^1_P(H^*X,H^*X)=0$. Then there is an injection
\[
Aut_{\mathbb{K}}(H^*X)/Aut_P(H^*X)\hookrightarrow \pi_0 Real_{P_{\infty}}(H^*X)
\]
and an isomorphism
\[
\pi_1( Real_{P_{\infty}}(H^*X),\phi)\cong \{id_{(H^*X,\phi)}\}.
\]
For $n\geq 2$ we have
\begin{eqnarray*}
\pi_n( Real_{P_{\infty}}(H^*X),\phi) & \cong & \pi_n(\mathcal{N}\mathcal{R}^{P_{\infty}},(H^*X,\phi)) \\
 & \cong & \pi_{n-1}(haut_{P_{\infty}}(H_{\infty}),id) \\
 & \cong & H\Gamma_P^{-n}(H^*X,H^*X).
\end{eqnarray*}
where $H_{\infty}$ is a cofibrant resolution of $H^*X$ in $P_{\infty}$-algebras.
\end{thm}

\subsection*{Applications to Poincaré duality}

We apply these results to determine equivalence classes of realizations of Poincaré duality on oriented compact manifolds.
Let $M$ be a compact oriented manifold. Then $H^*(M;\mathbb{K})$ is a particular kind of Frobenius algebra, called a special symmetric Frobenius bialgebra
(a structure which plays a prominent role in the algebraic counterpart of conformal field theory \cite{FRS2}). Let us denote by $sFrob$ the prop encoding such bialgebras, we have
\[
Aut_{\mathbb{K}}(H^*X)/Aut_{Com}(H^*X)\hookrightarrow \pi_0sFrob_{\infty}\{H^*X\}_{[\phi]}
\]
where $\phi$ is the trivial $sFrob_{\infty}$-structure $sFrob_{\infty}\stackrel{\sim}{\rightarrow}sFrob
\rightarrow End_{H^*X}$.
To put these observations into words, we prove that we can compute
connected components of the local realization space of Poincaré duality on cochains around the trivial homotopy structure as a quotient of its graded vector space automorphisms by its commutative algebra automorphisms.
We apply this to compute subsets of connected components of realizations of Poincaré duality for some examples of manifolds:
\begin{itemize}
\item{(1)} For the $n$-sphere $S^n$ we have
\[
\mathbb{Q}^*\hookrightarrow \pi_0sFrob_{\infty}\{H^*S^n\}_{[\phi]}.
\]

\item{(2)} For the complex projective space $\mathbb{C}P^n$ we have
\[
(\mathbb{Q}^*)^{\times n}\hookrightarrow \pi_0sFrob_{\infty}\{H^*\mathbb{C}P^n\}_{[\phi]}.
\]

\item{(3)}  Let $S$ be a compact oriented surface of genus $g$. We have
\[
\mathbb{Q}^*\hookrightarrow \pi_0sFrob_{\infty}\{H^*S\}_{[\phi]}.
\]

\item{(4)} For $M=\mathbb{C}P^2 \sharp \mathbb{C}P^2$, that is, the connected sum of two copies of $\mathbb{C}P^2$, we have
\[
GL_2(\mathbb{Q})/CO_2(\mathbb{Q})\times\mathbb{Q}^*\hookrightarrow \pi_0sFrob_{\infty}\{H^*M\}_{[\phi]}
\]
where $CO_2(\mathbb{Q})=\mathbb{Q}^*\ltimes O_2(\mathbb{Q})$ is the conformal orthogonal group.
\end{itemize}
This shows that there is an infinite number of such realizations with a certain subset of connected components parametrized by quotients of well known groups.
\begin{rem}
Let us note that by \cite{CMW} there is a non trivial action of the Grothendieck-Teichmüller Lie algebra
on the higher homotopy groups of realization spaces of homotopy Frobenius structures.
\end{rem}

\noindent
\textit{Organization}:
Section 1 is a reminder about props and their homotopical properties, algebras over props, and moduli spaces of such structures.
Section 2 is devoted to realizations spaces and the proof of Theorem 0.1. We define realization spaces and their decomposition into local realization spaces,
then we show how to replace the homotopy fiber sequence in the main theorem of \cite{Yal2} (restricted to local realization spaces) by an actual Kan fibration
obtained by a certain simplicial Borel construction (Theorem 2.8). From Theorem 2.8, \cite[Corollary 2.18]{Yal4} and a long exact sequence argument
we get Theorem 0.1 (Theorem 2.14 in Section 2.3). In Section 3, we treat the case of algebras over operads, for which an explicit computation of connected
components of the whole realization space as well as higher homotopy groups can be obtained in terms of the $\Gamma$-cohomology defined in \cite{Hof}.
Section 4 deals with realizations of extensions of a given algebraic structure. The results of section 2 extend to the relative case by proving Theorem 0.2
and an analogue of Theorem 2.8, from which we deduce a characterization of connected components of relative realization spaces analogue to part (1) of Theorem 0.1.
Finally, Section 5 focuses on an application of such results to the realization of Poincaré duality at the cochain level.
We prove several results about connected components of realizations of special symmetric Frobenius bialgebras, as well as extensions of a commutative algebra structure
to such a bialgebra structure, and conclude with some examples.

\medskip

\noindent
\textbf{Convention.} Throughout the text, we work in the category $Ch_{\mathbb{K}}$ of $\mathbb{Z}$-graded cochain complexes over a field $\mathbb{K}$.

\medskip

\noindent
\textbf{Funding statement.} This article was written while the author was employed as an assistant researcher at the University of Luxembourg, and the current version was completed while the author was funded by the European Union's Horizon 2020 research and innovation programme under the Marie Sklodowska-Curie grant agreement No 704589 at the University of Copenhagen.

\medskip

\noindent
\textbf{Acknowledgements.} I would like to thank David Chataur for useful discussions about realization problems, which inspired the writing of the present paper.

\section{Classification spaces and moduli spaces}

The category $Ch_{\mathbb{K}}$ of cochain complexes over a field $\mathbb{K}$ is our main working example of symmetric monoidal model category, that is, a model category equipped with a compatible symmetric monoidal structure. We assume that the reader is familiar with the basics of model categories, and we refer to Hirschhorn \cite{Hir} and Hovey \cite{Hov} for a comprehensive
treatment. We refer the reader to \cite[Chapter 4]{Hov} for the axioms of symmetric monoidal model categories
formalizing the interplay between the tensor and the model structure. Such a compatibility between the tensor product and the model category structure is crucial to make dg props inherit good homotopical properties, in particular to get well defined notions of homotopy algebra structures and moduli spaces of such structures.

\subsection{On $\Sigma$-bimodules, props and algebras over a prop}

A (differential graded, dg for short) $\Sigma$-biobject is a double sequence $\{M(m,n)\in Ch_{\mathbb{K}}\}_{(m,n)\in\mathbb{N}^{2}}$ where each $M(m,n)$ is equipped with a right action of $\Sigma_{m}$ and a left action of $\Sigma_{n}$ commuting with each other.
\begin{defn}
(1) A dg prop is a $\Sigma$-biobject endowed with associative horizontal products
\[
\circ_{h}:P(m_{1},n_{1})\otimes P(m_{2},n_{2})\rightarrow P(m_{1}+m_{2},n_{1}+n_{2}),
\]
vertical associative composition products
\[
\circ_{v}:P(k,n)\otimes P(m,k)\rightarrow P(m,n)
\]
and units $1\rightarrow P(n,n)$ neutral for both composition products.
These products satisfy the exchange law
\[
(f_1\circ_h f_2)\circ_v(g_1\circ_h g_2) = (f_1\circ_v g_1)\circ_h(f_2\circ_v g_2)
\]
and are compatible with the actions of symmetric groups.
Morphisms of props are equivariant morphisms of collections compatible with the composition products.
We denote by $\mathcal{P}$ the category of props.

(2) A prop $P$ has non-empty inputs if it satisfies
\[
P(0,n)=\begin{cases}
\mathbb{K}, & \text{if $n=0$},\\
0 & \text{otherwise}.
\end{cases}
\]
We define in a symmetric way a prop with non-empty outputs. We denote by $\mathcal{P}_0$ the full subcategory of $\mathcal{P}$ consisting of props with non-empty inputs.
\end{defn}
The following definition shows how a given prop encodes algebraic operations on the tensor powers
of a cochain complex:
\begin{defn}
(1) The endomorphism prop of a cochain complex $X$ is given by $End_{X}(m,n)=Hom_{Ch_{\mathbb{K}}}(X^{\otimes m},X^{\otimes n})$
where $Hom_{Ch_{\mathbb{K}}}$ is the internal hom bifunctor of $Ch_{\mathbb{K}}$.

(2) Let $P$ be a dg prop. A $P$-algebra is a cochain complex $X$ equipped with a prop morphism $P\rightarrow End_{X}$.
\end{defn}
Hence any ``abstract'' operation of $P$ is sent to an operation on $X$, and the way abstract operations
compose under the composition products of $P$ tells us the relations satisfied by the corresponding
algebraic operations on $X$.

There is a free-forgetful adjunction between $\Sigma$-biobjects and props, for which we refer to \cite{Fre2}, which transfer the cofibrantly generated model category structure of $\Sigma$-biobjects to the category of props:
\begin{thm}(see \cite{Fre2}, theorem 4.9)
(1) Suppose that $char(\mathbb{K})>0$. The category $\mathcal{P}_{0}$
of props with non-empty inputs (or outputs) equipped with the
classes of componentwise weak equivalences and componentwise fibrations forms a cofibrantly generated semi-model category.

(2) Suppose that $char(\mathbb{K})=0$. Then the entire category of props inherits a full cofibrantly generated model category
structure with the weak equivalences and fibrations as above.
\end{thm}
\medskip{}
A semi-model category structure is a slightly weakened version of model category structure: the
lifting axioms work only with cofibrations with a cofibrant domain, and the factorization axioms work only
on a map with a cofibrant domain (see the relevant section of \cite{Fre2}).
The notion of a semi-model category is sufficient to apply the usual constructions of homotopical algebra.
Let us note that props with non-empty inputs or non-empty outputs usually encode algebraic structures without unit or without counit, for instance Lie bialgebras.

\subsection{Moduli spaces and relative moduli spaces of algebra structures over a prop}

A moduli space of algebra structures over a prop $P$, on a given cochain complex $X$, is a simplicial set whose points are
the prop morphisms $P\rightarrow End_{X}$. For the sake of brevity and clarity, we refer the reader to the chapter 16 in \cite{Hir} for a complete treatment of the notions of simplicial resolutions, cosimplicial resolutions
and Reedy model categories.

\begin{defn}
Let $P$ be a cofibrant dg prop (with non-empty inputs if $char(\mathbb{K})>0$) and $X$ be a cochain complex.
The moduli space of $P$-algebra structures on $X$ is the simplicial set alternatively defined by
\[
P\{X\} = Mor_{\mathcal{P}_{0}}(P\otimes\Delta[-],End_X).
\]
where $P\otimes\Delta[-]$ is a cosimplicial resolution of $P$ and $\mathcal{P}_{0}$ is the category of props with non-empty inputs (see Definition 1.1),
or
\[
P\{X\} = Mor_{\mathcal{P}_{0}}(P,End_X^{\Delta[-]}).
\]
where $End_X^{\Delta[-]}$ is a simplicial resolution of $End_X$.
Since every cochain complex over a field is fibrant and cofibrant, every dg prop is fibrant:
the fact that $P$ is cofibrant and $End_X$ is fibrant implies that these
two formulae give the same moduli space up to homotopy.

Let us note that if the field $\mathbb{K}$ is of characteristic zero, then one can replace $\mathcal{P}_0$ by the category $\mathcal{P}$ of all dg props in this definition.
\end{defn}
We can already get two interesting properties of these moduli
spaces:
\begin{prop}
(1) The simplicial set $P\{X\}$ is a Kan complex and its connected components give the equivalences classes
of $P$-algebra structures on $X$, i.e
\[
\pi_0P\{X\}\cong [P,End_X]_{Ho(\mathcal{P}_0)}.
\]

(2) Every weak equivalence of cofibrant props $P\stackrel{\sim}{\rightarrow}Q$ gives
rise to a homotopy equivalence of fibrant simplicial sets $Q\{X\}\stackrel{\sim}{\rightarrow}P\{X\}$.
\end{prop}

These properties directly follow from the properties of simplicial mapping spaces in model categories \cite{Hir}.
The higher simplices of these moduli spaces encode higher simplicial homotopies between homotopies.

Now let us see how to study homotopy structures with respect to a fixed one.
Let $O\rightarrow P$ be a morphism of props inducing a morphism of cofibrant props
$O_{\infty}\rightarrow P_{\infty}$ between the corresponding cofibrant resolutions.
Let $X$ be a cochain complex, and suppose that $X$ possesses a $O_{\infty}$-algebra structure given by a morphism $\psi:O_{\infty}\rightarrow End_X$. The extensions of this $O_{\infty}$-algebra structure into $P_{\infty}$-algebra structures are parametrized by the homotopy fiber
\[
P_{\infty}\{X,\psi\} \rightarrow P_{\infty}\{X\}\stackrel{i^*}{\rightarrow}O_{\infty}\{X\},
\]
where $i^*$ is the map induced by precomposing with $i$ and the fiber is taken over the base point $\psi$. The Kan complex $P_{\infty}\{X,\psi\}$ is our relative moduli space of $P_{\infty}$-algebra structures on $X$ extending its fixed $O_{\infty}$-algebra structure $\psi$.
\begin{example}
Our main example of interest is the following. We know that the $E_{\infty}$-algebra structures on the singular
cochains classify the rational homotopy type of the considered topological space.
Let $X$ be a Poincaré duality space, then its cohomology form a Frobenius algebra (and thus a Frobenius bialgebra, see Section 5). Its singular cochains $C_*(X;\mathbb{Q})$ equipped with the higher cup products form an $E_{\infty}$-algebra quasi-isomorphic to the Sullivan model of $X$ in rational homotopy theory. Let us note $\psi:E_{\infty}\rightarrow End_{C_*(X;\mathbb{Q})}$ this $E_{\infty}$-algebra structure. This Sullivan model is a commutative dg algebra whose cohomology forms a Frobenius bialgebra, hence by \cite{LS} is quasi-isomorphic, as a commutative algebra, to a certain Frobenius bialgebra. This Frobenius bialgebra structure can then be transferred along the two quasi-isomorphisms above into a homotopy Frobenius bialgebra structure on singular cochains extending the $E_{\infty}$-structure, by using the relative transfer theorem given by Theorem 4.1 below, which means that the relative moduli space $Frob_{\infty}\{C_*(X;\mathbb{Q}),\psi\}$ is not empty.
A way to understand the homotopy Frobenius structures up to the rational homotopy type of this space is thus
to analyze $\pi_*Frob_{\infty}\{C_*(X;\mathbb{Q}),\psi\}$.
\end{example}

\subsection{Moduli spaces as homotopy fibers}

We recall from \cite{Yal2} a commutative diagram which will be crucial in the remaining part of the paper, and refer the reader to \cite{Yal2} for more details.
Let $\mathcal{P}$ a cofibrant prop, and $\mathcal{N} w(\mathcal{E}^{cf})^{\Delta[-]\otimes P}$
the bisimplicial set defined by $(\mathcal{N} w(\mathcal{E}^{cf})^{\Delta[-]\otimes P})_{m,n}=(\mathcal{N} w(\mathcal{E}^{cf})^{\Delta[n]\otimes P})_{m}$,
where the $w$ denotes the subcategory of morphisms which are weak equivalences in $\mathcal{E}$.
Then we get a diagram
\[
\xymatrix{P\{X\}\ar[r]\ar[d] & diag\mathcal{N} fw(\mathcal{E}^{cf})^{\Delta[-]\otimes P}\ar[d]\ar[r]^{\sim} & diag\mathcal{N} w(\mathcal{E}^{cf})^{\Delta[-]\otimes P} & \mathcal{N} w(\mathcal{E}^{cf})^{P}\ar[l]_-{\sim}\ar[d]\\
pt\ar[r] & \mathcal{N}(fw\mathcal{E}^{cf})\ar[rr]^{\sim} &  & \mathcal{N}(w\mathcal{E}^{cf})
},
\]
where the $fw$ denotes the subcategory of morphisms which are acyclic fibrations in $\mathcal{E}$.
The crucial point here is that the left-hand commutative square of this diagram is a homotopy pullback, implying that
we have a homotopy pullback of simplicial sets (see \cite[Theorem 0.1]{Yal2})
\[
\xymatrix{P\{X\}\ar[d]\ar[r] & \mathcal{N}(wCh_{\mathbb{K}}^P)\ar[d]\\
\{X\}\ar[r] & \mathcal{N}wCh_{\mathbb{K}}.
}
\]

\section{Characterization of realization spaces via the Borel construction}

\subsection{Fibrations and the simplicial Borel construction}

We denote by $sSet_*$ the category of pointed simplicial sets.
First recall that any Kan fibration
\[
F\stackrel{j}{\rightarrow} K \stackrel{p}{\rightarrow} L,
\]
where we suppose $(L,l_0)$ pointed and $F$ is the fiber over
this basepoint $l_0$ (the typical fiber), induces a long exact sequence
\[
...\pi_{n+1}(L,l_0)\stackrel{\partial_{n+1}}{\rightarrow} \pi_n(F,k_0)\stackrel{\pi_n(j)}{\rightarrow}\pi_n(K,k_0)
\stackrel{\pi_n(p)}{\rightarrow}\pi_n(L,l_0)\stackrel{\partial_n}{\rightarrow}...
\]
where $l_0=p(k_0)$. The boundary map $\partial_1:\pi_1(L)\rightarrow \pi_0(F)$ factors through a bijection $\pi_1(L)/Im(\pi_1(p))\cong Im(\partial_1)$, hence an injection
\[
\pi_1(L)/Im(\pi_1(p))\hookrightarrow \pi_0F
\]
which will be useful later to determine subsets of connected components of our realization spaces.
Let us note that $Im(\partial_1)$ is not in bijection with $\pi_0F$ in general but just a subset of it. By exact sequence arguments, the map $\partial_1$ is a surjection for instance when $L$ and $K$ are connected, but not anymore if $K$ is not connected.

Simplicial monoids are monoid objects in $sSet_*$ (or $sSet$).
Let $G$ be a simplicial monoid and $K$ be a simplicial set. A left action of $G$ on $K$ is a simplicial map
\begin{eqnarray*}
G_{\bullet}\times K_{\bullet} & \rightarrow & K_{\bullet} \\
(g,x) & \mapsto & g.x \\
\end{eqnarray*}
such that $g_2.(g_1.x)=(g_2g_1).x$ and $e.x=x$ where $e$ is the neutral element of $G_{\bullet}$.
One can associate to $G$ two particular simplicial sets $W_{\bullet}G$
and $\overline{W}_{\bullet}G$, the later being the classifying complex of $G$.
The set $\overline{W}_0G$ has a unique element, and $\overline{W}_nG=G_{n-1}\times G_{n-2}\times...\times G_0$
for $n>0$. We refer to \cite[Chapter IV, pages 87-88]{May} for the formulae defining the faces and the degeneracies.
The space $WG$ is the decalage of $\overline{W}G$, i.e $W_nG=\overline{W}_{n+1}G$ and the faces and degeneracies are shifted as well.
One can see that $WG=G\times_{\tau(G)}\overline{W}G$ where $\tau(G)(g_{n-1},...,g_0)=g_{n-1}$ is a twisting function
$\overline{W}G\rightarrow G$. The associated projection $p_G:WG\rightarrow\overline{W}G$ is a principal $G$-bundle
called the universal $G$-bundle, and $WG$ is contractible.

Now let $K$ be a simplicial set endowed with an action of $G$, we associate to it a twisted cartesian product
$K\times_{\tau}\overline{W}G$ where $\tau$ is a twisting function and $(K\times_{\tau}\overline{W}G)_n=K_n\times G_n$.
This twisted cartesian product actually equals to the quotient $K\times_G WG=K\times WG/((g.x,y)=(x,g.y))$.
One can construct from the principal $G$-bundle $p_G$ a $G$-bundle $p_G^*:K\times_G WG\rightarrow \overline{W}G$
of fiber $K$ defined by $p_G^*(x,y)=p_G(x)$.
By \cite[Chapter IV, Proposition 8.4]{May}, every twisted cartesian product whose fiber is a Kan complex
forms a Kan fibration, and by \cite[Chapter IV, Theorem 19.4]{May}, the $G$-bundles are twisted cartesian products,
so a $G$-bundle with a Kan complex as fiber forms a Kan fibration.
To conclude, let us recall that there is a bijective correspondence between homotopy classes of simplicial maps
$B\rightarrow \overline{W}G$ and $G$-equivalence classes of $G$-bundles of base $B$ with a fixed fiber \cite[Chapter IV, Theorem 21.13]{May}.

\subsection{Realization spaces and automorphisms groups}

The general idea is the following. Let $P$ be an arbitrary prop. We consider a cochain complex $X$ such that $H^*X$ is equipped with a $P$-algebra structure.
We would like to study the $P_{\infty}$-algebra structures $\varphi:P_{\infty}\rightarrow End_X$ on $X$ such that
$H^*(X,\varphi)\cong H^*X$ as $P$-algebras, that is the $P_{\infty}$-algebra structures realizing the $P$-algebra $H^*X$
at the cochain level.

\noindent
\textbf{Assumption}: we suppose that the differential of $P$ is trivial, since we need prop isomorphisms
$H^*P_{\infty}\cong H^*P\cong P$.

Any prop morphism $\varphi:P_{\infty}\rightarrow End_X$ induces a prop morphism $P\rightarrow End_{H^*X}$ in the following way.
By \cite[Proposition 3.4.7]{Fre1}, any symmetric monoidal dg functor $F:Ch_{\mathbb{K}}\rightarrow Ch_{\mathbb{K}}$ induces, for every $X\in Ch_{\mathbb{K}}$, a prop morphism $End_X\rightarrow End_{F(X)}$ defined for $m,n\in\mathbb{N}$ by
\[
Hom_{Ch_{\mathbb{K}}}(X^{\otimes m},X^{\otimes n}) \rightarrow Hom_{Ch_{\mathbb{K}}}(F(X^{\otimes m}),F(X^{\otimes n})) \cong Hom_{Ch_{\mathbb{K}}}(F(X)^{\otimes m},F(X)^{\otimes n}).
\]
Since $H^*(-)$ is a symmetric monoidal functor on cochain complexes, there exists
a prop morphism $\rho_X:End_X\rightarrow End_{H^*X}$. We thus obtain a prop morphism
\[
P_{\infty}\stackrel{\varphi}{\rightarrow} End_X\stackrel{\rho_X}{\rightarrow} End_{H^*X}
\]
and apply the cohomology functor aritywise
\[
H^*(\rho_X\circ\varphi):P\cong H^*P_{\infty}\rightarrow H^*End_{H^*X}=End_{H^*X}.
\]

A way to encode these structures, their equivalences classes and higher homotopies is to construct a realization space
$Real_{P_{\infty}}(H^*X)$ which will be a certain Kan subcomplex of $P_{\infty}\{X\}$.
For this aim we need the following lemma:
\begin{lem}
Let $K$ be a Kan complex. Let $L_0\subset K_0$ be a subset of the vertices of $K$.
There exists a unique Kan subcomplex $L\subset K$, constructed by induction on the simplicial dimension,
such that for any $n\geq 0$, the $n$-simplices of $L$ are the $n$-simplices of $K$ whose faces belong to $L_{n-1}$.
\end{lem}
\begin{proof}
We define the simplices of $L$ by induction on the simplicial dimension:
\[
L_1=\{k\in K_1|d_0k,d_1k\in L_0\}
\]
...
\[
L_n=\{k\in K_n|\forall i, d_ik\in L_{n-1}\}.
\]
By construction, the faces of $K$ restrict to $L$. We now check by induction on the simplicial dimension $n$ that the degeneracies of $K$ also restrict to $L$.
For $n=0$, let $x$ be a vertex of $L$. There is only one degeneracy $s_0(x)\in K_1$, and by the simplicial identities we have
\[
(d_0\circ s_0)(x)=x \in L_0
\]
and
\[
(d_1\circ s_0)(x)=x \in L_0,
\]
so $s_0(x)\in L_1$ since its faces belong to $L_0$.
Now, let us suppose that there exists a positive integer $n$ such that the degeneracies $s_j:K_{n-1}\rightarrow K_n$ restrict to maps $s_j:L_{n-1}\rightarrow L_n$.
Let $\sigma\in L_n$ be a $n$-simplex of $L$ and $0\leq j\leq n$ be a fixed integer.
For $i=j$ and $i=j+1$, we have
\[
(d_i\circ s_j)(\sigma)=\sigma\in L_n.
\]
For $i<j$ we have
\[
(d_i\circ s_j)(\sigma)=(s_{j-1}\circ d_i)(\sigma)
\]
which belongs to $L_n$, because by definition of $L_n$ the $(n-1)$-simplex $d_i(\sigma)$ belongs to $L_{n-1}$, and by induction hypothesis
$s_{j-1}:K_{n-1}\rightarrow K_n$ restricts to $s_{j-1}:L_{n-1}\rightarrow L_n$.
For $i>j+1$ we have
\[
(d_i\circ s_j)(\sigma)=(s_j\circ d_{i-1})(\sigma)
\]
which belongs to $L_n$ by the same arguments.
All the faces of $s_j(\sigma)$ belongs to $L_n$, so $s_j(\sigma)$ belongs to $L_{n+1}$.

Regarding the Kan condition, it is sufficient to check it in dimension $2$ (it works similarly in higher dimensions,
given the inductive construction of $L$). The horn $\Lambda^2$ consists in two $1$-simplices having one common vertex
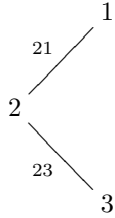
\[
\xymatrix{
 & 1 \\ 2 \ar@{-}[ur]^{21} \ar@{-}[dr]_{23} & \\ & 3
}
\]
Consider the image of $\Lambda^2$ in $L$ under a given simplicial map.
Since $L$ is included into a Kan complex $K$, one can fill this horn in $K$, choosing a $1$-simplex $13$ relating
the simplices $1$ and $3$ and a $2$-simplex $123$ filling the triangle.
Now, since $1$ and $3$ are in $L$, by definition $13$ also lies in $L$. Given that $21$, $23$ and $13$
are $1$-simplices of $L$, by construction $123$ is a $2$-simplex of $L$ and the horn is filled in $L$.
\end{proof}
We apply this construction to the Kan complex $P_{\infty}\{X\}$, by choosing as a set of vertices
the $P_{\infty}$-algebra structures $P_{\infty}\rightarrow End_X$ realizing $P\rightarrow End_{H^*X}$.
We obtain the desired Kan complex $Real_{P_{\infty}}(H^*X)$.

\subsubsection{Local realization spaces}

The reader will notice that our definition of the realization space is not the same as the one given for instance in \cite{GH},
where it is defined as the simplicial nerve of a category whose objects are the $P_{\infty}$-algebras with a cohomology isomorphic
to $H^*X$ as $P$-algebras, and morphisms are quasi-isomorphisms of $P_{\infty}$-algebras.
One the one hand, we restrict ourselves to the $P_{\infty}$-algebras having $X$ as underlying complex.
On the other hand, according to \cite{Fre2}, if two prop morphisms $\varphi_1,\varphi_2:P_{\infty}\rightarrow End_X$
are homotopic then $(X,\phi_1)$ and $(X,\phi_2)$ are related by a zigzag of quasi-isomorphisms of $P_{\infty}$-algebras.
But the converse is in general not true, so our realization space detects finer homotopical informations about the
$P_{\infty}$-algebra structures on $X$ than the one of \cite{GH}.
This difference is precisely measured by subspaces of $Real_{P_{\infty}}(H^*X)$ that we call local realization spaces,
which also provide us a decomposition of $Real_{P_{\infty}}(H^*X)$ into more handable pieces.
\begin{defn}
Let $\phi:P_{\infty}\rightarrow End_X$ be a prop morphism realizing a fixed structure $P\rightarrow End_{H^*X}$.
The space of local realizations of $H^*X$ at $\phi$ is the Kan subcomplex of $P_{\infty}\{X\}$ denoted by
$P_{\infty}\{X\}_{[\phi]}$ defined by
\[
(P_{\infty}\{X\}_{[\phi]})_0=\{\varphi:P_{\infty}\rightarrow End_X|\exists (X,\varphi)\stackrel{\sim}{\leftarrow}
\bullet\stackrel{\sim}{\rightarrow} (X,\phi)\in Ch_{\mathbb{K}}^{P_{\infty}}\}
\]
where the zigzags are of any finite length.
The higher simplices are defined by the construction explained previously.
\end{defn}

\begin{lem}
A zigzag of $P_{\infty}$-algebras $(X,\varphi_X)\stackrel{\sim}{\leftarrow}\bullet\stackrel{\sim}{\rightarrow}(Y,\varphi_Y)$
induces an isomorphism of $P$-algebras $H^*X\cong H^*Y$.
\end{lem}
\begin{proof}
Let us consider such a zigzag $(X,\varphi_X)\stackrel{\sim}{\leftarrow}\bullet\stackrel{\sim}{\rightarrow}(Y,\varphi_Y)$.
The $P_{\infty}$-algebra structure on the whole diagram is encoded by a prop morphism
\[
P_{\infty}\rightarrow End_{X\stackrel{\sim}{\leftarrow}\bullet\stackrel{\sim}{\rightarrow}Y}.
\]
For every small category $I$, the cohomology defines obviously a functor $H^*:Ch_{\mathbb{K}}^I\rightarrow Ch_{\mathbb{K}}^I$
which associates to a functor $F:I\rightarrow Ch_{\mathbb{K}}$ the composite $H^*\circ F$.
This functor is symmetric monoidal, since the tensor product on diagrams is defined pointwise.
Moreover, the category of diagrams $Ch_{\mathbb{K}}^I$ is a symmetric monoidal category over $Ch_{\mathbb{K}}$.
Hence for any $F\in Ch_{\mathbb{K}}^I$ one can form the prop morphism $\rho_F:End_F\rightarrow End_{H^*F}$,
and any prop morphism $\varphi:P_{\infty}\rightarrow End_F$ induces the prop morphism $H^*(\rho_F\circ\varphi):
P\rightarrow End_{H^*F}$. The conclusion follows from the fact that the endomorphism prop of $F$ that we consider in this proof,
defined by using the external hom of the diagram category $Ch_{\mathbb{K}}^I$, is exactly the same prop as
the endomorphism prop of diagrams in the sense of \cite{Fre2} which is used to encode $P_{\infty}$-algebra structures
on diagrams. We apply this argument to the particular case of
$(X,\varphi_X)\stackrel{\sim}{\leftarrow}\bullet\stackrel{\sim}{\rightarrow}(Y,\varphi_Y)$.
\end{proof}

\begin{cor}
The local realization space $P_{\infty}\{X\}_{[\phi]}$ is a Kan subcomplex of the realization space $Real_{P_{\infty}}(H^*X)$.
\end{cor}

We could define more generally $P_{\infty}\{X\}_{[\phi]}$ for any morphism $\phi$, it will be then the local realization space
of $H^*(\rho_X\circ\phi)$ at $\phi$.

Now we intend to prove that moduli spaces and realization spaces decompose in a disjoint union of these local realization spaces.
\begin{lem}
Let $\phi,\phi':P_{\infty}\rightarrow End_X$ be two prop morphisms. We denote by $[X,\phi]$ and $[X,\phi']$ the weak equivalence
classes respectively of $(X,\phi)$ and $(X,\phi')$ in $P_{\infty}$-algebras. We have

(1) $P_{\infty}\{X\}_{[\phi]}=P_{\infty}\{X\}_{[\phi']}$ or $P_{\infty}\{X\}_{[\phi]}\bigcap P_{\infty}\{X\}_{[\phi']}=\emptyset$;

(2) $P_{\infty}\{X\}_{[\phi]}=P_{\infty}\{X\}_{[\phi']}$ if and only if $[X,\phi]=[X,\phi']$.
\end{lem}
\begin{proof}
(1) Suppose that $P_{\infty}\{X\}_{[\phi]}\bigcap P_{\infty}\{X\}_{[\phi']}\neq\emptyset$.
Let $\varphi$ be a common vertex of $P_{\infty}\{X\}_{[\phi]}$ and $P_{\infty}\{X\}_{[\phi']}$,
there exists two zigzags $(X,\varphi)\stackrel{\sim}{\leftarrow}\bullet\stackrel{\sim}{\rightarrow}(X,\phi)$
and $(X,\varphi)\stackrel{\sim}{\leftarrow}\bullet\stackrel{\sim}{\rightarrow}(X,\phi')$,
hence a zigzag $(X,\phi)\stackrel{\sim}{\leftarrow}\bullet\stackrel{\sim}{\rightarrow}(X,\phi')$.
This zigzag implies, by definition, that $P_{\infty}\{X\}_{[\phi]}$ and $P_{\infty}\{X\}_{[\phi]}$ share
the same set of vertices. The inductive construction of $P_{\infty}\{X\}_{[\phi]}$ and $P_{\infty}\{X\}_{[\phi']}$
then implies that they share the same set of simplices in each dimension.

(2) Obvious by using (1).
\end{proof}
Let us denote by $\mathcal{N}$ the simplicial nerve functor, $\mathcal{N}wCh_{\mathbb{K}}^{P_{\infty}}$ the classification
space of $P_{\infty}$-algebras with respect to quasi-isomorphisms, and $\mathcal{N}\mathcal{R}^{P_{\infty}}$ a version of the realization
space in the sense of \cite{GH} restricted to the complex $X$. That is, the nerve of the category $\mathcal{R}^{P_{\infty}}$ whose
objects are $P_{\infty}$-algebras $(X,\phi)$ with underlying complex $X$ such that $H^*(X,\phi)\cong H^*X$ as $P$-algebras and morphisms are
weak equivalences of $P_{\infty}$-algebras.
We deduce from the lemma above:
\begin{prop}
The moduli space $P_{\infty}\{X\}$ decomposes into
\[
\coprod_{[X,\phi]\in\pi_0\mathcal{N}wCh_{\mathbb{K}}^{P_{\infty}}} P_{\infty}\{X\}_{[\phi]}
\]
The realization space $Real_{P_{\infty}}(H^*X)$ decomposes into
\[
\coprod_{[X,\phi]\in\pi_0\mathcal{N}\mathcal{R}^{P_{\infty}}} P_{\infty}\{X\}_{[\phi]}
\]
\end{prop}

Before showing how to compute the homotopy groups of these local realization spaces, let us point out that in the definition
of these spaces, we could have replaced the zigzags of quasi-isomorphisms by zigzags of acyclic fibrations.
This technical point will be useful in the next proofs:
\begin{lem}
Let $P_{\infty}\{X\}_{[\phi]}^f$ be the Kan subcomplex of the realization space defined by
\[
(P_{\infty}\{X\}_{[\phi]})_0=\{\varphi:P_{\infty}\rightarrow End_X|\exists (X,\varphi)\stackrel{\sim}{\twoheadleftarrow}
\bullet\stackrel{\sim}{\twoheadrightarrow} (X,\phi)\in Ch_{\mathbb{K}}^{P_{\infty}}\}
\]
where the zigzags are of any finite length. We have
\[
P_{\infty}\{X\}_{[\phi]}^f=P_{\infty}\{X\}_{[\phi]}.
\]
\end{lem}
\begin{proof}
By construction, it is sufficient to prove that these two subcomplexes share the same set of vertices.
The vertices of $P_{\infty}\{X\}_{[\phi]}^f$ are included in the set of vertices of $P_{\infty}\{X\}_{[\phi]}$,
so we just have to prove the converse inclusion.
Let $\varphi$ be a vertex of $P_{\infty}\{X\}_{[\phi]}$. The $P_{\infty}$-algebra $(X,\varphi)$
is linked to $(X,\phi)$ by a zigzag of weak equivalences of finite length.
We want to prove that this zigzag can be replaced by a zigzag of acyclic fibrations.
It is sufficient to prove it in two cases: the case of a direct arrow and the case of two opposite arrows.
Once these two cases work, the method can be extended by induction to all zigzags.

\textit{First case.} Let $f:(X,\varphi)\stackrel{\sim}{\rightarrow}(X,\phi)$ be a weak equivalence of $P_{\infty}$-algebras.
Its image under the forgetful functor admits a decomposition $X\stackrel{i}{\rightarrow}Z(X)\stackrel{p}{\rightarrow} X$
in cochain complexes into an acyclic cofibration followed by an acyclic fibration. Since $X$ is fibrant (like any cochain complex
over a field), the map $i$ admits a section $s:Z(X)\stackrel{\sim}{\twoheadrightarrow} X$ such that $s\circ i = id_X$,
which is surjective and thus forms an acyclic fibration.
We obtain the following commutative diagram:
\[
\mathcal{Y}(X):\xymatrix{ &  & X\\
X\ar@/^{1pc}/[urr]^{=}\ar@/^{-1pc}/[drr]_f
\ar@{>->}[]!R+<4pt,0pt>;[r]_-{i}^-{\sim} & Z(X)\ar@{->>}[ur]_{s}^{\sim}\ar@{->>}[dr]^{p}_{\sim}\\
 &  & X
}.
\]
We know that there exists a prop morphism $P_{\infty}\rightarrow End_{X\stackrel{=}{\leftarrow}X
\stackrel{f}{\rightarrow} X}$, so we can apply the argument of \cite[Lemma 8.3]{Fre2} to get
a lifting
\[
\xymatrix{
 & End_{\mathcal{Y}(X)} \ar@{->>}[d]^-{\sim} \\ P_{\infty} \ar[r] \ar[ur] & End_{X\stackrel{=}{\leftarrow}X
\stackrel{f}{\rightarrow} X}
}
\]
(since $P_{\infty}$ is a cofibrant prop) which preserves $\varphi$ and $\phi$,
hence a zigzag $(X,\varphi)\stackrel{\sim}{\twoheadleftarrow}Z(X)\stackrel{\sim}{\twoheadrightarrow}(X,\phi)$
of $P_{\infty}$-algebra implying that $\varphi\in P_{\infty}\{X\}_{[\phi]}^f$.

\textit{Second case.}
Suppose that there is a zigzag of two oppposite weak equivalences
$(X,\varphi)\stackrel{g}{\leftarrow}(X',\varphi')\stackrel{d}{\rightarrow}(X,\phi)$.
The map of cochain complexes $(g,d):X\rightarrow X\times X$ factors through
an acyclic cofibration followed by a fibration
$X\stackrel{\sim}{\rightarrow}Z(X)\twoheadrightarrow X\times X$ hence a commutative diagram
\[
\mathcal{Y}(X):\xymatrix{ &  & X\\
X\ar@/^{1pc}/[urr]^g\ar@/^{-1pc}/[drr]_d
\ar@{>->}[]!R+<4pt,0pt>;[r]^-{\sim} & Z(X)\ar@{->>}[ur]^{\sim}\ar@{->>}[dr]_{\sim}\\
 &  & X
}.
\]
Since there is a prop morphism $P_{\infty}\rightarrow End_{X\stackrel{g}{\leftarrow}X
\stackrel{d}{\rightarrow} X}$ we can apply the same argument again to obtain a lifting
$P_{\infty}\rightarrow End_{\mathcal{Y}(X)}$ and conclude that $\varphi\in P_{\infty}\{X\}_{[\phi]}^f$.
\end{proof}

\subsubsection{A characterization via homotopy automorphisms}

Recall from Section 1.3 that $P_{\infty}\{X\}$ can be identified as a homotopy fiber
\[
\xymatrix{P_{\infty}\{X\}\ar[r]\ar[d] & diag\mathcal{N} fwCh_{\mathbb{K}}^{P\otimes\Delta^{\bullet}}\ar[d]\\
\{X\}\ar[r] & \mathcal{N}fwCh_{\mathbb{K}} \sim  \mathcal{N}wCh_{\mathbb{K}}
}.
\]
This homotopy fiber restricts to a homotopy fiber
\[
\xymatrix{P_{\infty}\{X\}_{[\phi]}^f\ar[r]\ar[d] & diag\mathcal{N} fwCh_{\mathbb{K}}^{P\otimes\Delta^{\bullet}}|_X\ar[d]\\
\{X\}\ar[r] &  \mathcal{N}wCh_{\mathbb{K}}|_X
}
\]
where $\mathcal{N}wCh_{\mathbb{K}}|_X$  is the connected component of the cochain complex $X$
and
\[
diag\mathcal{N} fwCh_{\mathbb{K}}^{P\otimes\Delta^{\bullet}}|_X
=\coprod_{[X,\varphi]} diag\mathcal{N} fwCh_{\mathbb{K}}^{P\otimes\Delta^{\bullet}}|_{(X,\varphi)}
\]
is the union of the connected components $diag\mathcal{N} fwCh_{\mathbb{K}}^{P\otimes\Delta^{\bullet}}|_{(X,\varphi)}$
of the $(X,\varphi)$ ranging over the acyclic fibration classes of $P_{\infty}$-algebras
having $X$ as underlying complex.
According to the previous lemma, one gets actually the homotopy fiber
\[
\xymatrix{P_{\infty}\{X\}_{[\phi]}\ar[r]\ar[d] & diag\mathcal{N} fwCh_{\mathbb{K}}^{P\otimes\Delta^{\bullet}}|_X\ar[d]\\
\{X\}\ar[r] &  \mathcal{N}wCh_{\mathbb{K}}|_X
}.
\]
The main goal of this section is to prove the following theorem:
\begin{thm}
Let us suppose that $X$ is a cochain complex with the trivial differential (e.g $X$ is the cohomology of some cochain complex).
Let $\phi:P_{\infty}\rightarrow End_X$ be a prop morphism.
There exists a commutative square
\[
\xymatrix{
WAut_{\mathbb{K}}(X)\times_{Aut_{\mathbb{K}}(X)} P_{\infty}\{X\}_{[\phi]}^f \ar[r]^-{\sim} \ar@{->>}[d]_{\pi}
& diag\mathcal{N} fwCh_{\mathbb{K}}^{P\otimes\Delta^{\bullet}}|_X\ar[d]\\
\overline{W}Aut_{\mathbb{K}}(X) \ar[r]^-{\sim} & \mathcal{N}wCh_{\mathbb{K}}|_X
}
\]
where $\pi$ is a Kan fibration obtained by the simplicial Borel construction
and the horizontal maps are weak equivalences of simplicial sets.
\end{thm}

Since $\pi$ is obtained by applying the simplicial Borel construction to a simplicial
action of $Aut_{\mathbb{K}}(X)$ (seen as a discrete group) on $P_{\infty}\{X\}_{[\phi]}^f$,
we get:
\begin{prop}
The strict pullback below is a homotopy pullback of simplicial sets:
\[
\xymatrix{
P_{\infty}\{X\}_{[\phi]}^f \ar[d] \ar[r] & WAut_{\mathbb{K}}(X)\times_{Aut_{\mathbb{K}}(X)} P_{\infty}\{X\}_{[\phi]}^f \ar@{->>}[d]_{\pi}\\
\{X\} \ar[r] & \overline{W}Aut_{\mathbb{K}}(X)
}.
\]
\end{prop}
\begin{proof}
Since $Aut_{\mathbb{K}}(X)$ is a discrete simplicial group acting on $P_{\infty}\{X\}_{[\phi]}^f$,
according to \cite{May} we can associate to this action an $Aut_{\mathbb{K}}(X)$-bundle of typical fiber $P_{\infty}\{X\}_{[\phi]}^f$ defined by $\pi$. Moreover, $P_{\infty}\{X\}_{[\phi]}^f$
is a Kan complex so $\pi$ is a Kan fibration. Given that the model category of simplicial sets
is right proper, the fiber of $\pi$ is a model of the homotopy fiber of $\pi$.
\end{proof}

\textit{Proof of Theorem 2.8}

We first need to define the action of $Aut_{\mathbb{K}}(X)$ on the local realization spaces:
\begin{lem}
The group $Aut_{\mathbb{K}}(X)$ acts simplicially on $P_{\infty}\{X\}_{[\phi]}^f$.
\end{lem}
\begin{proof}
We see $Aut_{\mathbb{K}}(X)$ as a discrete simplicial group and we construct a simplicial action on
the whole moduli space by conjugation in each simplicial dimension
\begin{eqnarray*}
Aut_{\mathbb{K}}(X)\times P_{\infty}\{X\}_k & \rightarrow & P_{\infty}\{X\}_k \\
(g:X\stackrel{\cong}{\rightarrow}X,\varphi:P\otimes\Delta^k\rightarrow End_X) & \mapsto &
\varphi_g=g^{-1}_*\circ g^*\circ\varphi
\end{eqnarray*}
where
\begin{eqnarray*}
\varphi_g(m,n):(P\otimes\Delta^k)(m,n) & \stackrel{\varphi}{\rightarrow} & Hom_{\mathbb{K}}(X^{\otimes m},X^{\otimes n})\\
 & \stackrel{(g^{\otimes m})^*}{\rightarrow} & Hom_{\mathbb{K}}(X^{\otimes m},X^{\otimes n})\\
 & \stackrel{((g^{-1})^{\otimes n})^*}{\rightarrow} & Hom_{\mathbb{K}}(X^{\otimes m},X^{\otimes n}).
\end{eqnarray*}
The notation $Hom_{\mathbb{K}}$ stands for the internal hom of cochain complexes.
The map $g^{-1}_*\circ g^*$ is a prop morphism:

-regarding the horizontal composition products, for any $f_1\in Hom_{\mathbb{K}}(X^{\otimes m_1},X^{\otimes n_1})$
and $f_2\in Hom_{\mathbb{K}}(X^{\otimes m_2},X^{\otimes n_2})$, we have
\begin{eqnarray*}
(g^{-1}_*\circ g^*)(f_1\otimes f_2) & = & (g^{-1})^{\otimes n_1+n_2}\circ f_1\otimes f_2
\circ g^{\otimes m_1+m_2} \\
 & = & ((g^{-1})^{\otimes n_1}\circ f_1\circ g^{\otimes m_1})\otimes ((g^{-1})^{\otimes n_2}\circ f_2\circ g^{\otimes m_2}) \\
  & = & (g^{-1}_*\circ g^*)(f_1)\otimes (g^{-1}_*\circ g^*)(f_2);
\end{eqnarray*}

-regarding the vertical composition products, for any $f_1\in Hom_{\mathbb{K}}(X^{\otimes m},X^{\otimes k})$
and $f_2\in Hom_{\mathbb{K}}(X^{\otimes m_2},X^{\otimes n_2})$, we have
\begin{eqnarray*}
(g^{-1}_*\circ g^*)(f_2\circ f_1) & = & (g^{-1})^{\otimes n}\circ f_2\circ f_1\circ g^{\otimes m}\\
 & = & (g^{-1})^{\otimes n}\circ f_2\circ (g\circ g^{-1})^{\otimes k} \circ f_1\circ g^{\otimes m}\\
 & = & (g^{-1})^{\otimes n}\circ f_2\circ g^{\otimes k}\circ (g^{-1})^{\otimes k} \circ f_1\circ g^{\otimes m}\\
 & = & (g^{-1}_*\circ g^*)(f_2)\circ (g^{-1}_*\circ g^*)(f_1).
\end{eqnarray*}

The compatibility with the actions of symmetric groups is trivial since it amounts to
permute the factors in $g^{\otimes m}$ and $(g^{-1})^{\otimes n}$.
By construction we clearly get a group action, that is $\varphi_{g_2\circ g_1}=(\varphi_{g_1})_{g_2}$
and $\varphi_{id}=\varphi$.

We check the compatibility of these actions defined in each simplicial dimension
with the simplicial structure of $P_{\infty}\{X\}$.
We have a commutative diagram
\[
\xymatrix{
P_{\infty}\otimes\Delta^k \ar[r]^-{\varphi} & End_X \ar[d]^{=}\ar[r]^-{\cong}_-{g^{-1}_*\circ g^*}
 & End_X \ar[d]^{=} \\
P_{\infty}\otimes\Delta^{k-1} \ar[u]\ar[r]_-{d_i\varphi} & End_X \ar[r]^-{\cong}_-{g^{-1}_*\circ g^*}
 & End_X
}
\]
so $d_i(\varphi_g)=d_i(\varphi_g)$.
The same argument holds for the degeneracies.

Finally, we have to verify that this action restricts to the local realization space
$P_{\infty}\{X\}_{[\phi]}^f$. We prove it inductively on the simplicial dimension $k$:
since a $k$-simplex of the moduli space $P_{\infty}\{X\}$ lies in $P_{\infty}\{X\}_{[\phi]}^f$
if and only if all its vertices lie in $P_{\infty}\{X\}_{[\phi]}^f$, it is sufficient to verify
the restriction of the group action for $k=0$.
Let $\varphi:P_{\infty}\rightarrow End_X$ be a vertex of $P_{\infty}\{X\}_{[\phi]}^f$, since
$g^{-1}_*\circ g^*$ is a prop isomorphism there is a homotopy equivalence $\varphi_g\sim\varphi$.
According to \cite{Fre2}, this homotopy equivalence implies the existence of a zigzag of acyclic fibrations
(not only weak equivalences) between $(X,\varphi_g)$ and $(X,\varphi)$,
hence between $(X,\varphi_g)$ and $(X,\phi)$. Therefore $\varphi_g\in P_{\infty}\{X\}_{[\phi]}^f$.
\end{proof}

With this well defined simplicial group action we start the proof of Theorem 2.8.
First we construct the commutative diagram
\[
\xymatrix{
WAut_{\mathbb{K}}(X)\times_{Aut_{\mathbb{K}}(X)} P_{\infty}\{X\}_{[\phi]}^f \ar[r] \ar@{->>}[d]_{\pi}
& diag\mathcal{N} fwCh_{\mathbb{K}}^{P\otimes\Delta^{\bullet}}|_X\ar[d]\\
\overline{W}Aut_{\mathbb{K}}(X) \ar[r] & \mathcal{N}wCh_{\mathbb{K}}|_X
}
\]
in the following way:
\[
\xymatrix{
((f_k,...,f_0),\varphi:P_{\infty}\otimes\Delta^k\rightarrow End_X) \ar[r] \ar[d]_{\pi=p^*_{Aut_{\mathbb{K}}(X)}} &
((X\phi)\stackrel{\cong}{\rightarrow}(X,f_k.\varphi)...\stackrel{\cong}{\rightarrow}(X,(f_k\circ...\circ f_1).\varphi)) \ar[d]^{forget} \\
(f_{k-1},...,f_0) \ar[r] & (X\stackrel{f_{k-1}}{\rightarrow}...\stackrel{f_0}{\rightarrow}X)
}
\]
where the left vertical map is the projection associated to the Borel construction
and the right vertical map forgets the $P_{\infty}\otimes\Delta^k$-algebra structure.
The top horizontal map transfers the $P_{\infty}\otimes\Delta^k$-algebra structure on $X$
along the sequence of isomorphisms given by $f_k,...,f_0$ and the bottom horizontal map
is just an inclusion.
It is clear by definition of faces and degeneracies in the simplicial structures involved
that these four maps are simplicial.

It remains to prove that the two horizontal maps are weak equivalences.
For the bottom arrow, it follows from the work of Dwyer-Kan \cite{DK3} which identifies
the connected components of the classification space of a model category with
the classifying complexes of homotopy automorphism and from the fact that
$\overline{W}haut(X)=\overline{W}Aut_{\mathbb{K}}(X)$.

For the top arrow, we have a morphism of homotopy fibers
\[
\xymatrix{
P_{\infty}\{X\}_{[\phi]}^f \ar[r]\ar[d]^{=} & WAut_{\mathbb{K}}(X)\times_{Aut_{\mathbb{K}}(X)} P_{\infty}\{X\}_{[\phi]}^f \ar[d] \ar[r] &
\overline{W}Aut_{\mathbb{K}}(X)\ar[d]^{\sim} \\
P_{\infty}\{X\}_{[\phi]}^f \ar[r] & diag\mathcal{N} fwCh_{\mathbb{K}}^{P\otimes\Delta^{\bullet}}|_X \ar[r] &
\mathcal{N}wCh_{\mathbb{K}}|_X \\
}
\]
which induces a morphism of long exact sequences of homotopy groups. Since the two external arrows
of the diagram above induce isomorphisms of homotopy groups, we conclude that the middle one is a
weak equivalence by applying the five lemma.

\begin{rem}

(1) One can apply the same arguments to obtain similar results for $P_{\infty}\{X\}$ and
$Real_{P_{\infty}}(X)$.

(2) One can study for instance realization problems of diagrams of bialgebras.
Let $I$ be a small category, the category of $I$-diagrams $Ch_{\mathbb{K}}^I$ is equipped with
the injective model category structure (pointwise cofibrations) and with the structure of
a symmetric monoidal category over $Ch_{\mathbb{K}}$. These two structures are compatible,
so that $Ch_{\mathbb{K}}^I$ forms a monoidal model category over $Ch_{\mathbb{K}}$.
We can therefore define $P$-algebras and
$P_{\infty}$-algebras in these diagrams. The cohomology induces a symmetric monoidal functor
$H^*:Ch_{\mathbb{K}}^I\rightarrow Ch_{\mathbb{K}}^I$, so for any $F\in Ch_{\mathbb{K}}^I$,
a morphism $P_{\infty}\rightarrow End_F$ induces $P\rightarrow H^*End_X\rightarrow End_{H^*F}$
and it is then meaningful to study realization problems in this setting.
For instance, when $I=\{\bullet\rightarrow\bullet\}$ one can consider realizations of morphisms
of $P$-algebras.
\end{rem}

\subsubsection{When $X$ has a non-zero differential.}

We recall from \cite{Fre2} that the endomorphism prop $End_f$ of a morphism $f:X\rightarrow Y$ can be obtained
by an explicit pullback in $\Sigma$-biobjects
\[
\xymatrix{
End_f\ar[r]^{d_1}\ar[d]^{d_0} & End_Y\ar[d]^{f^*} \\
End_X\ar[r]_{f_*} & Hom_{XY}
}
\]
where the $\Sigma$-biobject $Hom_{XY}$ is defined by $Hom_{XY}(m,n)=Hom_{\mathbb{K}}(X^{\otimes m},X^{\otimes n})$,
and the maps $f^*$ and $f_*$ are defined respectively by $f^*(m,n)=-\circ f^{\otimes m}$ and $f_*(m,n)=f^{\otimes n}\circ-$.
In this pullback square the maps $d_0$ and $d_1$ are prop morphisms.
According to \cite[Lemma 7.2]{Fre2}, if $f$ is an acyclic fibration then $d_1$ is an acyclic
fibration and $d_0$ a weak equivalence.
Given that every dg prop is fibrant, the zigzag of weak equivalences of fibrant props
\[
End_X\stackrel{\sim}{\leftarrow}End_f\stackrel{\sim}{\rightarrow} End_Y
\]
from the pullback above gives rise by postcomposition to a zigzag of weak equivalences of Kan complexes
\[
P_{\infty}\{X\}\stackrel{\sim}{\leftarrow}P_{\infty}\{f\}\stackrel{\sim}{\rightarrow} P_{\infty}\{Y\}.
\]
In particular, when working over a field, there is a projection $X\twoheadrightarrow H^*X$ of a cochain complex
$X$ to its cohomology which forms a quasi-isomorphism.
It implies that any $P_{\infty}$-algebra structure on $H^*X$ can be transferred to $X$
via this acyclic fibration (see \cite[Theorem 7.3(1)]{Fre2}).
Then we apply the argument above to obtain a zigzag of weak equivalences between
$P_{\infty}\{X\}$ and $P_{\infty}\{H^*X\}$:
\begin{prop}
Suppose that the moduli space $P_{\infty}\{H^*X\}$ is not empty. Then the moduli space
$P_{\infty}\{X\}$ is not empty and has the same homotopy type.
\end{prop}
When $H^*X$ is equipped with a $P$-algebra structures, it possesses at least the trivial $P_{\infty}$-algebra structure given by $P_{\infty}\stackrel{\sim}{\rightarrow}P\rightarrow End_{H^*X}$ so the proposition above applies.

The homotopy equivalence above restricts to local realization spaces:
\begin{prop}
Let $f:(X,\phi_X)\stackrel{\sim}{\twoheadrightarrow} (Y,\phi_Y)$ be an acyclic fibration of $P_{\infty}$-algebras. Then there is a zigzag of morphisms Kan complexes
\[
P_{\infty}\{X\}_{[\phi_X]}\leftarrow P_{\infty}\{f\}_{[\phi_f]}\rightarrow P_{\infty}\{Y\}_{[\phi_Y]}
\]
induced by $d_0$ and $d_1$,
where $\phi_f:P_{\infty}\rightarrow End_f$ is the structure of $P_{\infty}$-algebra morphism of $f$. This zigzag induces
isomorphisms between the $n^{th}$ homotopy groups for $n>0$, based at any point homotopic to $\phi$, and an injection of connected components
\[
\pi_0 P_{\infty}\{Y\}_{[\phi_Y]}\hookrightarrow \pi_0 P_{\infty}\{X\}_{[\phi_X]}.
\]
\end{prop}
\begin{proof}
Let $\varphi$ be a vertex of $P_{\infty}\{f\}_{[\phi_f]}$. There exists a zigzag of weak equivalences
$(f,\varphi)\stackrel{\sim}{\leftarrow}\bullet\stackrel{\sim}{\rightarrow}(f,\phi_f)$ in
$(Ch_{\mathbb{K}}^{\bullet\rightarrow\bullet})^{P_{\infty}}$, that is a commutative diagram
\[
\xymatrix{
X\ar[d]^{(f,\varphi)} & \bullet\ar[l]_-{\sim}\ar[r]^-{\sim} & X\ar[d]^{(f,\phi_f)} \\
Y & \bullet\ar[l]_-{\sim}\ar[r]^-{\sim} & Y }.
\]
This diagram induces two zigzags
\[
(d_0)_*(f,\varphi)=(X,d_0\circ\varphi)\stackrel{\sim}{\leftarrow}\bullet\stackrel{\sim}{\rightarrow}
(X,d_0\circ\phi_f)=(d_0)_*(f,\phi_f)=(X,\phi_X)
\]
and
\[
(d_1)_*(f,\varphi)=(Y,d_1\circ\varphi)\stackrel{\sim}{\leftarrow}\bullet\stackrel{\sim}{\rightarrow}
(Y,d_1\circ\phi_f)=(d_1)_*(f,\phi_f)=(Y,\phi_Y)
\]
so $(d_0)_*(f,\varphi)\in P_{\infty}\{X\}_{[\phi_X]}$
and $(d_1)_*(f,\varphi)\in P_{\infty}\{Y\}_{[\phi_Y]}$.
Hence the zigzag
\[
P_{\infty}\{X\}\stackrel{\sim}{\leftarrow}P_{\infty}\{f\}\stackrel{\sim}{\rightarrow} P_{\infty}\{Y\}.
\]
restricts to a zigzag
\[
P_{\infty}\{X\}_{[\phi_X]}\leftarrow P_{\infty}\{f\}_{[\phi_f]}\rightarrow P_{\infty}\{Y\}_{[\phi_Y]}.
\]
Since we have
\[
\pi_n(P_{\infty}\{X\}_{[\phi_X]},\psi) = \pi_n(P_{\infty}\{X\},\psi)
\]
for $n>0$ and $\psi$ homotopic to $\phi$, the maps $(d_0)_*$ and $(d_0)_*$ induces isomorphisms of higher homotopy groups between the local realization spaces.

Now we want to check that $\pi_0(d_0)_*$ is an injection and $\pi_0(d_1)_*$ a bijection.
The surjectivity of $\pi_0(d_1)_*$ follows from the lifting diagram
\[
\xymatrix{
 & End_f \ar@{->>}[d]_{\sim}^{d_1}\\
P_{\infty}\ar[r]_-{\varphi} \ar[ur]& End_Y
}
\]
for any $\varphi\in _{\infty}\{Y\}_{[\phi_Y]}$, in which the lift exists because $d_1$ is an acyclic fibration and $P_{\infty}$ is cofibrant.
For the injectivity, we have to prove that given $\varphi_f,\varphi_f':P_{\infty}\rightarrow End_f \in P_{\infty}\{f\}_{[\phi_f]}$,
if $d_1\circ\varphi_f$ and $d_1\circ\varphi_f'$ are homotopic then $\varphi_f$ and $\varphi_f'$ are homotopic.
By \cite[Proposition 2.5]{Yal4}, the componentwise tensor product $End_Y\otimes A_{PL}(\Delta^1)$ (where $A_{PL}$ is Sullivan's functor of piecewise linear forms in rational homotopy theory \cite{Sul}) gives a path object on $End_X$, so a homotopy between
$d_1\circ\varphi_f$ and $d_1\circ\varphi_f'$ is given by a map
\[
H:P_{\infty}\rightarrow End_Y\otimes A_{PL}(\Delta^1).
\]
This map lifts to a map
\[
\xymatrix{
 & End_f \ar@{->>}[d]_{\sim}^{d_1\otimes A_{PL}(\Delta^1)}\\
P_{\infty}\ar[r]_-H \ar[ur]^{H_f}& End_Y \otimes A_{PL}(\Delta^1)
}.
\]
Indeed, acyclic fibrations of props are determined componentwise in cochain complexes, and the tensor product of cochain complexes preserves acyclic surjections,
so $d_1\otimes A_{PL}(\Delta^1)$ is an acyclic fibration. The lift $H_f$ is the desired homotopy.

The injectivity of $\pi_0(d_0)_*$ follows from the following argument.
The homotopy fiber $hofib_{\varphi}((d_0)_*)$ of $(d_0)_*:P_{\infty}\{f\}\stackrel{\sim}{\rightarrow}P_{\infty}\{X\}$ over a point $\varphi:P_{\infty}\rightarrow End_X$
is the simplicial set generated by the lifts
\[
\xymatrix{
 & End_f \ar[d]_{\sim}^{d_0}\\
P_{\infty}\ar[r]_-{\varphi} \ar[ur]& End_X
}
\]
along the weak equivalence $d_0$.
By a general homotopical algebra argument, the set of lifts along a weak equivalence is connected.
Indeed, if $\varphi_f,\varphi_f':P_{\infty}\rightarrow End_f$ are two lifts of $\varphi$, then in the homotopy category of props $Ho(\mathcal{P})$ we have
\[
[d_0]\circ [\varphi_f]=[d_0\circ\varphi_f]=[d_0\circ\varphi_f']=[d_0]\circ [\varphi_f']
\]
where $[-]$ stands for the homotopy class of a map, that is, its image under the localization functor $\mathcal{P}\rightarrow Ho(\mathcal{P})$.
Since $d_0$ is a weak equivalence, the homotopy class $[d_0]$ is an isomorphism in $Ho(\mathcal{P})$, hence
\[
[\varphi_f]=[\varphi_f'].
\]
This means that $\varphi_f$ and $\varphi_f'$ belongs to the same connected component of $hofib_{\varphi}((d_0)_*)$.
Considering the exact sequence of pointed sets
\[
\pi_0hofib_{\varphi}((d_0)_*)\rightarrow \pi_0P_{\infty}\{f\}\stackrel{\pi_0(d_0)_*}{\rightarrow}\pi_0P_{\infty}\{X\}
\]
induced by the homotopy fiber sequence
\[
hofib_{\varphi}((d_0)_*)\rightarrow P_{\infty}\{f\}\stackrel{\sim}{\rightarrow}P_{\infty}\{X\},
\]
this implies that $\pi_0(d_0)_*$ is injective.
\end{proof}
The zigzag
\[
P_{\infty}\{X\}\leftarrow P_{\infty}\{f\}\rightarrow P_{\infty}\{Y\}
\]
restricts to a zigzag between the realization spaces
\[
Real_{P_{\infty}}(H^*X)\leftarrow Real{P_{\infty}}(H^*f)\rightarrow Real_{P_{\infty}}(H^*Y).
\]
Indeed, if $\phi_f$ realizes $H^*f$, then $d_0\circ\phi_f$ and $d_1\circ\phi_f$ realize
respectively $\phi_X$ and $\phi_Y$ according to the following commutative diagrams:
consider
\[
\xymatrix{
End_f\ar@{->>}[r]^{d_0}\ar[d]_{\rho_f} & End_X\ar[d]^{\rho_X} \\
End_{H^*f}\ar@{->>}[r] & End_{H^*X}
}
\]
and apply the cohomology functor to obtain
\[
\xymatrix{
H^*End_f\ar@{->>}[r]^{H^*d_0}\ar[d] & H^*End_X\ar[d] \\
End_{H^*f}\ar@{->>}[r] & End_{H^*X}
}
\]
and finally
\[
\xymatrix{
P \ar[dr]^{H^*\phi_f} \ar@/^{1pc}/[drr]^{H^*\phi_X} \ar@/^{-1pc}/[ddr] & & \\
 & H^*End_f\ar@{->>}[r]^{H^*d_0}\ar[d] & H^*End_X\ar[d] \\
 & End_{H^*f}\ar@{->>}[r] & End_{H^*X}
}.
\]
The same argument holds for $Y$.

In particular, to understand the higher homotopy groups it is sufficient to study the realization space of $P\rightarrow End_{H^*X}$
in $P_{\infty}\rightarrow End_{H^*X}$, and the set of equivalences classes of realizations $P_{\infty}\rightarrow End_{H^*X}$ injects into
the set of equivalences classes of realizations $P_{\infty}\rightarrow End_X$.
For this, we introduce a variation of the realization space, denoted by $\overline{Real}_{P_{\infty}}(H^*X)$, which is a Kan complex whose vertices are the prop morphisms $P_{\infty}\rightarrow End_{H^*X}$ inducing the $P$-algebra structure $P\rightarrow End_{H^*X}$
in cohomology. Similarly to the realization space in Proposition 2.6, it admits a decomposition
\[
\overline{Real_{P_{\infty}}}(H^*X) = \coprod_{[\phi]\in\pi_0\mathcal{N}\mathcal{R}^{P_{\infty}}}P_{\infty}\{H^*X\}_{[\phi]}.
\]
By Proposition 2.13, there is an injection of connected components
\[
\pi_0\overline{Real}_{P_{\infty}}(H^*X)\hookrightarrow \pi_0Real_{P_{\infty}}(H^*X)
\]
and isomorphisms of homotopy groups
\[
\pi_n(\overline{Real}_{P_{\infty}}(H^*X),\varphi) \cong \pi_n(Real_{P_{\infty}}(H^*X),\varphi)
\]
for $n>0$.

\subsection{Realizations, homotopy automorphisms and obstruction theory}

Our main result describes the connected components of local realization spaces as quotient of
automorphisms groups and the higher homotopy groups as cohomology groups of a deformation complex.
Moreover, for $n\geq 2$ we recover the higher homotopy groups of the nerve $\mathcal{N}\mathcal{R}^{P_{\infty}}$.
\begin{thm}
For every $\phi:P_{\infty}\rightarrow End_{H^*X}$,
there is an injection
\[
Aut_{\mathbb{K}}(H^*X)/Aut_{Ho(Ch_{\mathbb{K}}^{P_{\infty}})}(H^*X,\phi)\hookrightarrow\pi_0P_{\infty}\{H^*X\}_{[\phi]}
\]
in the local realization space $P_{\infty}\{H^*X\}_{[\phi]}$ of $H^*X$ at $\phi$, and the higher homotopy groups are given by
\begin{eqnarray*}
\pi_1(P_{\infty}\{H^*X\}_{[\phi]},\phi) & \cong & ker(\pi_1(\pi)) \\
 & = & \pi_0Real_{P_{\infty}}(id_{(H^*X,\phi)}) \\
 & \cong & H^0Der_{\phi}(P_{\infty}, End_{H^*X})
\end{eqnarray*}
and for $n\geq2$,
\begin{eqnarray*}
\pi_n(P_{\infty}\{H^*X\}_{[\phi]},\phi) & \cong & \pi_n(\mathcal{N}\mathcal{R}^{P_{\infty}},(H^*X,\phi))\\
 & \cong & \pi_{n-1}(L^HwCh_{\mathbb{K}}^{P_{\infty}}((H^*X,\phi),(H^*X,\phi)),id) \\
 & \cong & H^{1-n}Der_{\phi}(P_{\infty}, End_{H^*X})\\
\end{eqnarray*}
where $L^H$ is the hammock localization functor constructed in \cite{DK2}.
\end{thm}

Let us just recall briefly that for any $P_{\infty}$-algebra $(X,\varphi)$, the hammocks of weight $k$ and length $n$, which define the $k$-simplices of the
simplicial monoid $L^HwCh_{\mathbb{K}}^{P_{\infty}}((X,\varphi),(X,\varphi))$ consist in the data of $k-1$ chains of morphisms of length $n$
\[
X\stackrel{f_0^1}{\leftarrow}\bullet...\bullet\stackrel{f_0^n}{\rightarrow}X,
\]
...,
\[
X\stackrel{f_k^1}{\leftarrow}\bullet...\bullet\stackrel{f_k^n}{\rightarrow}X
\]
organized in a commutative diagram (a ``hammock")
\[
\xymatrix{
 & C_{0,1}\ar[dl]_{f_0^1}\ar[d]^{...}\ar[r] & C_{0,2}\ar[d]^{...} &...\ar[l] &C_{0,n}\ar[l]_{f_0^{n-1}}\ar[d]^{...} \ar[dr]^{f_0^n} & \\
X & ... \ar[l]\ar[d]^{...}& ... \ar[d]^{...} & ...\ar[l] & ...\ar[d]^{...}\ar[r] & X \\
& C_{k,1}\ar[ul]^{f_n^1}\ar[r] & C_{k,2} & ...\ar[l] &C_{k,n}\ar[l]_{f_k^{n-1}} \ar[ur]_{f_0^n} & \\
}
\]
where vertical arrows and arrows going from right to left are weak equivalences.
The simplicial monoid $L^HwCh_{\mathbb{K}}^{P_{\infty}}((X,\varphi),(X,\varphi))$ encodes the homotopy automorphisms of $(X,\varphi)$.
When $P$ is an operad, then the $P_{\infty}$-algebras form a model category and we recover
the usual simplicial monoid of self weak equivalences $haut(X,\varphi)$.

\begin{cor}
There is an injection
\[
\coprod_{[\phi]\in\pi_0\mathcal{N}\mathcal{R}^{P_{\infty}}} Aut_{\mathbb{K}}(H^*X)/Aut_{Ho(Ch_{\mathbb{K}}^{P_{\infty}})}(H^*X,\phi)\hookrightarrow \pi_0Real_{P_{\infty}}(H^*X).
\]
For $n\geq 1$, the $n^{th}$ homotopy group of  $Real_{P_{\infty}}(H^*X)$ is given by
\begin{eqnarray*}
\pi_n(\overline{Real_{P_{\infty}}}(H^*X),\phi) & = & \pi_n(Real_{P_{\infty}}(H^*X),\phi) \\
 & = & \pi_n(P_{\infty}\{H^*X\}_{\phi},\phi) \\
 & = & \pi_n(P_{\infty}\{H^*X\}_{[\phi]},\phi) \\
 & \cong & H^{1-n}Der_{\phi}(P_{\infty},End_{H^*X})
\end{eqnarray*}
where $P_{\infty}\{H^*X\}_{\phi}$ is the connected component of $\phi$
in $P_{\infty}\{H^*X\}$.
\end{cor}
\begin{rem}
The long exact sequence associated to $P_{\infty}\{H^*X\}$ tells us that there is an injection
of $\pi_1(P_{\infty}\{H^*X\},\phi)$ into
\begin{eqnarray*}
\pi_1(\overline{W}L^HwCh_{\mathbb{K}}^{P_{\infty}}((H^*X,\phi),(H^*X,\phi)),id)
 & \cong & \pi_0L^HwCh_{\mathbb{K}}^{P_{\infty}}((H^*X,\phi),(H^*X,\phi)) \\
 & = & Aut_{Ho(Ch_{\mathbb{K}}^{P_{\infty}})}(H^*X,\phi)
\end{eqnarray*}
and that if $Aut(H^*X)=\{id\}$, then this injection is an isomorphism.
\end{rem}

\begin{proof}
There is a long exact sequence
\begin{eqnarray*}
...\rightarrow \pi_{n+1}(\overline{W}Aut_{\mathbb{K}}(H^*X),id) & \stackrel{\partial}{\rightarrow} &
\pi_n(P_{\infty}\{H^*X\}_{[\phi]}^f,\phi) \\
\rightarrow \pi_n(WAut_{\mathbb{K}}(H^*X)\times_{Aut_{\mathbb{K}}(H^*X)} P_{\infty}\{H^*X\}_{[\phi]}^f,(id,\phi))
& \stackrel{\pi_n(\pi)}{\rightarrow} & \pi_n(\overline{W}Aut_{\mathbb{K}}(H^*X),id)...
\end{eqnarray*}
Given that $Aut_{\mathbb{K}}(H^*X)$ is a discrete group, the higher homotopy groups of its classifying complex are
\[
\pi_1\overline{W}Aut_{\mathbb{K}}(H^*X)\cong Aut_{\mathbb{K}}(H^*X)
\]
and
\[
\pi_n\overline{W}Aut_{\mathbb{K}}(H^*X)=0
\]
otherwise.
Consequently, the group $\pi_1(P_{\infty}\{X\}_{[\phi]},\phi)$ is nothing but $ker(\pi_1(\pi))$,
which in turn corresponds to $\pi_0Real_{P_{\infty}}(id_{(H^*X,\phi)})$.

For $n\geq 2$,
\begin{eqnarray*}
\pi_n(P_{\infty}\{H^*X\}_{[\phi]}^f,\phi) & \cong & \pi_n(WAut_{\mathbb{K}}(H^*X)\times_{Aut_{\mathbb{K}}(H^*X)} P_{\infty}\{H^*X\}_{[\phi]}^f,(id,\phi)) \\
 & \cong & \pi_n(diag\mathcal{N}fwCh_{\mathbb{K}}^{P_{\infty}\otimes\Delta^{\bullet}}|_{H^*X},(H^*X,\phi)) \\
 & \cong & \pi_n(diag\mathcal{N}fwCh_{\mathbb{K}}^{P_{\infty}\otimes\Delta^{\bullet}}|_{(H^*X,\phi)},(H^*X,\phi)) \\
 & \cong & \pi_n(\mathcal{N}wCh_{\mathbb{K}}^{P_{\infty}}|_{H^*X},(H^*X,\phi)) \\
 & \cong & \pi_n(\overline{W}L^HwCh_{\mathbb{K}}^{P_{\infty}}((H^*X,\phi),(H^*X,\phi)),id)
\end{eqnarray*}
where the first line is a consequence of the long exact sequence, the second line follows from the comparison of fiber sequences of Theorem 2.8, the fourth line follows from \cite{Yal2}, and the fifth line is a result of Dwyer-Kan \cite{DK3}.
Recall that for any simplicial monoid $G$, the complex $\overline{W}G$ is a delooping of $G$,
that is, we have $\Omega_e\overline{W}G=G$ where $\Omega_e(-)$ is the loop space pointed at the neutral
element $e$. We deduce that
\[
\pi_{n+1}(\overline{W}G,e) \cong \pi_n(\Omega_e\overline{W}G,e) \cong \pi_n(G,e),
\]
hence
\[
\pi_n(\overline{W}L^HwCh_{\mathbb{K}}^{P_{\infty}}((H^*X,\phi),(H^*X,\phi)),id)
\cong \pi_{n-1}(L^HwCh_{\mathbb{K}}^{P_{\infty}}((H^*X,\phi),(H^*X,\phi)),id).
\]

We also get
\begin{eqnarray*}
 & \pi_1(WAut_{\mathbb{K}}(H^*X)\times_{Aut_{\mathbb{K}}(H^*X)} P_{\infty}\{H^*X\}_{[\phi]}^f,(id,\phi)) & \\
 \cong & \pi_1(\overline{W}L^HwCh_{\mathbb{K}}^{P_{\infty}}((H^*X,\phi),(H^*X,\phi)),id) & \\
 \cong & \pi_0(L^HwCh_{\mathbb{K}}^{P_{\infty}}((H^*X,\phi),(H^*X,\phi))) & \\
  = & Aut_{Ho(Ch_{\mathbb{K}}^{P_{\infty}})}(H^*X,\phi) &
\end{eqnarray*}
by equivalence of fiber sequences.
The image of $\pi_1(\pi)$ in $Aut_{\mathbb{K}}(H^*X)$ is the subgroup of graded automorphisms
which can be realized as a zigzag of quasi-isomorphisms
$(H^*X,\phi)\stackrel{\sim}{\leftarrow}\bullet\stackrel{\sim}{\rightarrow}(H^*X,\phi)$
in $Ch_{\mathbb{K}}^{P_{\infty}}$.
For convenience, we will denote the quotient $Aut_{\mathbb{K}}(H^*X)/Im(\pi_1(\pi))$ by $Aut_{\mathbb{K}}(H^*X)/Aut_{Ho(Ch_{\mathbb{K}}^{P_{\infty}})}(H^*X,\phi)$.
For $n=0$, the boundary
\[
\partial_1:\pi_1(\overline{W}Aut_{\mathbb{K}}(H^*X))=Aut_{\mathbb{K}}(H^*X)\rightarrow \pi_0P_{\infty}\{X\}_{[\phi]}
\]
gives the desired injection
\[
Aut_{\mathbb{K}}(H^*X)/Aut_{Ho(Ch_{\mathbb{K}}^{P_{\infty}})}(H^*X,\phi) \hookrightarrow\pi_0P_{\infty}\{X\}_{[\phi]}.
\]

To conclude, the identification of higher homotopy groups with the cohomology of deformation complexes follows from      \cite[Corollary 2.18]{Yal4}.
Although this result is originally stated for properads, its extension to props is straightforward by using the deformation complexes of algebras over props
introduced in \cite{Mar2}. Indeed, such deformation complexes possess a filtered $L_{\infty}$-algebra structure
to which the arguments of the proof of \cite[Theorem 2.7]{Yal4} apply as well.
\end{proof}

A crucial question in deformation theory is to know how to relate the $P_{\infty}$-algebra
homotopy automorphisms of $X$ with the strict $P$-algebra automorphisms of its cohomology $H^*X$.
For this, we define a map
\[
L^HwCh_{\mathbb{K}}^{P_{\infty}}(X,X)_0 \rightarrow Aut_P(H^*X)
\]
which sends any zigzag
\[
X\stackrel{f_1}{\leftarrow}\bullet\stackrel{f_2}{\rightarrow}...\stackrel{f_n}{\rightarrow} X
\]
on the composite $H^*f_n \circ ...\circ H^*f_2\circ (H^*f_1)^{-1}$
where every arrow going to the left is reversed by taking its inverse (since it is an isomorphism).
We check that this map factors through homotopy classes of vertices.
It is sufficient to verify it for zigzag of length two, since the argument in any length is the same.
Let $X\stackrel{f_0}{\leftarrow}\bullet\stackrel{g_0}{\rightarrow}X$ and
$X\stackrel{f_1}{\leftarrow}\bullet\stackrel{g_1}{\rightarrow}X$ be two such zigzags.
Suppose there is a homotopy
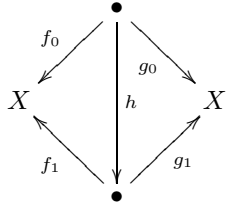
\[
\xymatrix{
 & \bullet \ar[ld]_{f_0} \ar[rd]_{g_0} \ar[dd]^h & \\
X & & X \\
 & \bullet \ar[ul]^{f_1} \ar[ur]_{g_1} &
}
\]
relating these two vertices.
Applying the cohomology functor $H^*(-)$ to this diagram, we obtain
$H^*f_0 = H^*f_1\circ H^*h$ and $H^*g_0=H^*g_1\circ h$, hence
\begin{eqnarray*}
H^*g_0\circ (H^*f_0)^{-1} & = & H^*g_1\circ H^*h\circ (H^*h)^{-1}\circ (H^*f_1)^{-1} \\
 & = & H^*g_1\circ (H^*f_1)^{-1}.
\end{eqnarray*}
This allows us to define a map at the level of connected components
\[
\overline{H^*}:\pi_0 L^HwCh_{\mathbb{K}}^{P_{\infty}}(X,X)=Aut_{Ho(Ch_{\mathbb{K}}^{P_{\infty}})}(X)
\rightarrow Aut_{Ho(Ch_{\mathbb{K}}^P)}(H^*X).
\]

\noindent
\textbf{Questions.}

(1) Is this map injective, surjective ?

(2) How to compute $\pi_0\mathcal{N}\mathcal{R}^{P_{\infty}}$ ?

\begin{rem}
If the differential of $X$ is zero, then this map is surjective. Indeed, if we consider $(X,\phi)\in
Ch_{\mathbb{K}}^P$, and $(X,\varphi)\in Ch_{\mathbb{K}}^{P_{\infty}}$, then the following diagram
commutes:
\[
\xymatrix{
 & End_{X_0\stackrel{\cong}{\rightarrow}X_1} \ar[d]_{\pi} \ar[r]^{\pi} & End_{X_1} \\
 P_{\infty}\ar[ur]^{s\circ\varphi} \ar[r]^{\varphi} & End_{X_1} \ar@/^1pc/[u]^s \ar@/_1pc/[ur]_{\pi\circ s=id}
}
\]
where $\pi$ is the canonical projection of the endomorphism prop of a diagram onto the endomorphism
prop of a subdiagram, which in this case is an isomorphism with inverse $s$.
\end{rem}

There is no well defined obstruction theory for algebras over a general prop.
It is thus difficult to get explicit obstruction groups for the connectedness of $\pi_0\mathcal{N}\mathcal{R}^{P_{\infty}}$,
and to know if $\overline{H^*}$ is injective. We can however simplify the computation of connected components of the local realization space of the trivial $P_{\infty}$-algebra structure (which are already non trivial), a result which will be useful in Section 5:
\begin{lem}
Let $P$ be a prop, and $X$ a cochain complex such that $H^*X$ forms a $P$-algebra.
Let $\phi:P_{\infty}\rightarrow P\rightarrow End_{H^*X}$ be the trivial $P_{\infty}$-algebra structure
on $H^*X$ induced by its $P$-algebra structure. We have
\[
Aut_{\mathbb{K}}(H^*X)/Aut_P(H^*X)\cong Aut_{\mathbb{K}}(H^*X)/Aut_{Ho(Ch_{\mathbb{K}}^{P_{\infty}})}(H^*X)\hookrightarrow\pi_0P_{\infty}\{H^*X\}_{[\phi]}.
\]
\end{lem}
\begin{proof}
We already know that $\pi_0P_{\infty}\{H^*X\}_{[\phi]}\cong Aut_{\mathbb{K}}(H^*X)/Aut_{Ho(Ch_{\mathbb{K}}^{P_{\infty}})}(H^*X)$.
We also know that $\overline{H^*}:Aut_{Ho(Ch_{\mathbb{K}}^{P_{\infty}})}(H^*X)\rightarrow Aut_P(H^*X)$
is surjective. It remains to prove that $\overline{H^*}$ is injective.
Let $H^*X\stackrel{f_0}{\leftarrow}Z_0\stackrel{g_0}{\rightarrow}H^*X$ and
$H^*X\stackrel{f_1}{\leftarrow}Z_1\stackrel{g_1}{\rightarrow}H^*X$ be two vertices of
$L^HwCh_{\mathbb{K}}^{P_{\infty}}(X,X)$ such that $H^*g_1\circ (H^*f_1)^{-1}=H^*g_0\circ (H^*f_0)^{-1}$
in $Aut_P(H^*X)$. We get the commutative diagrams
\[
\xymatrix{
H^*X \ar[d]_{=} & Z_i \ar[l]_-{f_i}^-{\sim}\ar[r]^-{g_i}_-{\sim}\ar@{->>}[d]_-{\sim}^-{p_{Z_i}}
& H^*X \ar[d]^{=} \\
H^*X & H^*Z_i \ar[l]_-{H^*f_i}^-{\cong}\ar[r]^-{H^*g_i}_-{\cong} & H^*X
}
\]
for $i=0,1$ in $Ch_{\mathbb{K}}$ where $p_{Z_i}$ is the projection, which forms here an acyclic fibration
of cochain complexes.
The map $f_i$ is a morphism of $P_{\infty}$-algebras for $H^*X$ equipped with $\phi$, and $H^*f_i$
an isomorphism of $P$-algebras thus an isomorphism of $P_{\infty}$-algebras for $H^*X$ and $H^*Z_i$
equipped with the trivial $P_{\infty}$-algebra structures. If we denote by $[-]$ the homotopy class
of a zigzag in $L^HwCh_{\mathbb{K}}^{P_{\infty}}(X,X)$, we get
\begin{eqnarray*}
[H^*X\stackrel{f_0}{\leftarrow}Z_0\stackrel{g_0}{\rightarrow}H^*X] & = &
[H^*X\stackrel{H^*f_0}{\leftarrow}H^*Z_0\stackrel{H^*g_0}{\rightarrow}H^*X]\\
 & = & [H^*X\stackrel{H^*f_1}{\leftarrow}H^*Z_1\stackrel{H^*g_1}{\rightarrow}H^*X] \\
 & = & [H^*X\stackrel{f_1}{\leftarrow}Z_0\stackrel{g_2}{\rightarrow}H^*X],
\end{eqnarray*}
using that the equality $H^*g_1\circ (H^*f_1)^{-1}= H^*g_0\circ (H^*f_0)^{-1}$
implies the existence of a commutative diagram
\[
\xymatrix{
 & H^*Z_0 \ar[ld]_{H^*f_0} \ar[rd]_{H^*g_0} \ar[dd] & \\
H^*X & & H^*X \\
 & H^*Z_1 \ar[ul]^{H^*f_1} \ar[ur]_{H^*g_1} &
}
\]
where the middle vertical map is defined by $(H^*f_1)^{-1}\circ H^*f_0 = (H^*g_1)^{-1}\circ H^*g_0$.
\end{proof}

\section{The operadic case}

Operads are used to parameterize various kind of algebraic structures.
Fundamental examples of operads include the operad $As$ encoding associative algebras,
the operad $Com$ of commutative algebras, the operad $Lie$ of Lie algebras and the operad
$Pois$ of Poisson algebras. There exists several equivalent approaches for the
definition of an algebra over an operad, for which we refer the reader to \cite{LV}.
Differential graded algebras over operads satisfy good homotopical properties:
\begin{thm}(see \cite{Fre2})
Let $P$ be a dg operad over a field $\mathbb{K}$ of characteristic zero.
The category of dg $P$-algebras $Ch_{\mathbb{K}}^P$ inherits a cofibrantly generated model category structure such that
a morphism $f$ of $P$-algebras is

(i)a weak equivalence if the underlying chain morphism is a quasi-isomorphism;

(ii)a fibration if it is degreewise injective;

(iii)a cofibration if it has the left lifting property with respect to acyclic fibrations.

\end{thm}

The model category structure allows one to express the internal symmetries of the objects as homotopy automorphisms.
In a model category $M$, the homotopy automorphisms of an object $X$ is the simplicial sub-monoid
$haut(X^{cf})\subset Map(X^{cf},X^{cf})$ of invertible connected components, where $X^{cf}$ is a cofibrant-fibrant resolution
of $X$,  i.e
\[
haut(X^{cf})=\coprod_{\overline{\phi}\in [X,X]_{Ho(M)}^{\times}} Map(X^{cf},X^{cf})_{\phi}
\]
where the $\overline{\phi}\in [X,X]_{Ho(M)}^{\times}$ are the automorphisms in the homotopy category of $M$
and $Map(X^{cf},X^{cf})_{\phi}$ the connected component of $\phi$ in the standard homotopy mapping space.

In this setting, we define a map
\[
\overline{H^*}:\pi_0 haut_{P_{\infty}}(X_{\infty})\rightarrow Aut_P(H^*X_{\infty})=Aut_P(H^*X)
\]
where $X_{\infty}\stackrel{\sim}{\rightarrow}X$ is a cofibrant resolution of $X$ in the category
of $P_{\infty}$-algebras (recall that all algebras are fibrant in cochain complexes over a field).
According to \cite[Theorem 2.7, Theorem 3.5]{Hof}, we know that obstructions to be injective or surjective
lie in the $\Gamma$-cohomology groups of $H^*X$ as a $P$-algebra:

(1) if $H\Gamma^0_P(H^*X,H^*X)=0$, then $\overline{H^*}$ is injective;

(2) if $H\Gamma^1_P(H^*X,H^*X)=0$, then $\overline{H^*}$ is surjective and $\pi_0\mathcal{N}\mathcal{R}^{P_{\infty}}(H^*X)=*$, in particular
$\pi_0\mathcal{N}\mathcal{R}^{P_{\infty}}=*$.

\begin{rem}
If $P$ is a $\Sigma$-cofibrant operad, $A$  a $P$-algebra and $M$ a $U_P(A)$-module (that is, a module over the enveloping algebra of $A$), then $H\Gamma_P^*(A,M)=H_P^*(A,M)$ is the usual operadic cohomology.
When $\mathbb{K}$ is a field of characteristic zero, all operads are $\Sigma$-cofibrant.
\end{rem}

\begin{thm}
Let $P$ be a $\Sigma$-cofibrant graded operad with a trivial differential. Let $H^*X$ be a $P$-algebra
such that $H\Gamma^0_P(H^*X,H^*X)=H\Gamma^1_P(H^*X,H^*X)=0$. Then there is an injection
\[
Aut_{\mathbb{K}}(H^*X)/Aut_P(H^*X)\hookrightarrow \pi_0 \overline{Real_{P_{\infty}}}(H^*X)
\]
and an isomorphism
\[
\pi_1( Real_{P_{\infty}}(H^*X),\phi)\cong \{id_{(H^*X,\phi)}\},
\]
where $\{id_{(H^*X,\phi)}\}$ is the trivial group with a unique element, the homotopy class of the identity map of $(H^*X,\phi)$(the realizations of the identity are all homotopic according to \cite{Hof}).
For $n\geq 2$ we have
\begin{eqnarray*}
\pi_n( Real_{P_{\infty}}(H^*X),\phi) & \cong & \pi_n(\mathcal{N}\mathcal{R}^{P_{\infty}},(H^*X,\phi)) \\
 & \cong & \pi_{n-1}(haut_{P_{\infty}}(H_{\infty}),id) \\
 & \cong & H\Gamma_P^{-n}(H^*X,H^*X).
\end{eqnarray*}
where $H_{\infty}$ is a cofibrant resolution of $H^*X$ in $P_{\infty}$-algebras.
\end{thm}
Note that the homotopy groups of $P_{\infty}\{H^*X\}$, hence of $P_{\infty}\{X\}$ by Proposition 2.12,
are the same for $n\geq 1$.

To conclude this section, we would like to emphasize a nice interpretation of the realization space $Real_{P_{\infty}}(H^*X)$.
According to \cite[Theorem 5.2.1]{Fre0}, left homotopies between two morphisms $\phi_0,\phi_1:P_{\infty}\rightarrow End_X$ corresponds to
$\infty$-isotopies between the two $P_{\infty}$-algebras $(X,\phi_0)$ and $(X,\phi_1)$, that is, $\infty$-morphisms of $P_{\infty}$-algebras
which reduce to the identity on $X$. Since the source of the morphisms $\phi_0$ and $\phi_1$ is cofibrant and their target is fibrant, left homotopies between
such maps are in bijection with right homotopies. The right homotopies are, in turn, the $1$-simplices of our realization space $Real_{P_{\infty}}(H^*X)$.
We conclude:
\begin{prop}
The Kan complex $Real_{P_{\infty}}(H^*X)$ parameterizes $P_{\infty}$-algebras with underlying complex $X$ realizing the $P$-algebra structure of $H^*X$
and their $\infty$-isotopies.
\end{prop}

\section{Extending algebraic structures: relative realization spaces}

Let $O\rightarrow P$ be a morphism of props, and suppose that $H^*X$ is equipped with a $P$-algebra
structure $P\rightarrow End_{H^*X}$. It induces an $O$-algebra structure $O\rightarrow P\rightarrow End_{H^*X}$.
We fix a $O_{\infty}$-algebra structure on $X$ which realizes this $O$-algebra structure on $H^*$.
Then a natural question is to determine the realizations $P_{\infty}\rightarrow End_X$ of the $P$-algebra structure
of $H^*X$ which extend this $O_{\infty}$-algebra structure on $X$.
Interesting examples are realizations of the Poincaré duality on the cochains of a manifold which extend the $E_{\infty}$-algebra
structure defined by the higher cup products (which defines the Wu classes, representing the Steenrod
squares on the cohomology).

The existence of a map $O\rightarrow P$ implies the existence of a map $O_{\infty}\rightarrow P_{\infty}$.
When such a map forms a cofibration, \cite[Theorem A]{Fre2} holds in the relative case:
\begin{thm}
Let $\mathcal{C}$ be a cofibrantly generated symmetric monoidal model category satisfying the monoid limit
axioms (see section 6 of \cite{Fre2}). Let $i:O_{\infty}\rightarrowtail P_{\infty}$ be a cofibration of cofibrant props
in $\mathcal{C}$. If $(X,\varphi_O^X)\stackrel{\sim}{\rightarrow} (Y,\varphi_O^Y)$ is a weak
equivalence of $O_{\infty}$-algebras such that $X,Y\in\mathcal{C}^{cf}$ (fibrant-cofibrant objects of
$\mathcal{C}$) and $Y$ is equipped with a $P_{\infty}$-algebra structure $\varphi_P^Y$ satisfying
$\varphi_P^Y\circ i=\varphi_O^Y$, then there exists a zigzag of two opposite acyclic fibrations of
$P_{\infty}$-algebras
\[
(X,\varphi_P^X)\stackrel{\sim}{\twoheadleftarrow}\bullet\stackrel{\sim}{\twoheadrightarrow}(Y,\varphi_P^Y)
\]
such that $\varphi_P^X\circ i=\varphi_O^X$.
\end{thm}
\begin{proof}
We just have to verify that the proof of \cite[Theorem A]{Fre2} can be transposed in the comma category $(\mathcal{P} \downarrow O_{\infty})$,
where $\mathcal{P}$ is the category of props equipped with the semi-model category structure defined in \cite{Fre2}.
Comma categories inherit the cofibrantly generated model category structure (see \cite{Hir2}), and this result
extends readily to the case of semi-model categories. Hence $(\mathcal{P} \downarrow O_{\infty})$ forms a cofibrantly generated
semi-model category. Now, let $P_{\infty}\in (\mathcal{P} \downarrow O_{\infty})$ such that $i:O_{\infty}\rightarrowtail P_{\infty}$
is a cofibration, then by definition of cofibrations and initial morphisms in comma categories,
the prop $P_{\infty}$ is cofibrant in $(\mathcal{P} \downarrow O_{\infty})$.
The existence of $\varphi_O^X$ and $\varphi_O^Y$ ensures that $End_X,End_Y\in (\mathcal{P} \downarrow O_{\infty})$,
and the condition $\varphi_P^Y\circ i=\varphi_O^Y$ ensures that $\varphi_P^Y$ is a morphism of $(\mathcal{P} \downarrow O_{\infty})$.
Now all the necessary conditions are satisfied and one applies exactly the same arguments as in \cite{Fre2}.
\end{proof}

Now let $i:O_{\infty}\rightarrow P_{\infty}$ be any fixed prop morphism, we do not suppose that $i$ is a cofibration
anymore. Let $X$ be a cochain complex, and suppose that $X$ possesses a $O_{\infty}$-algebra structure given by a morphism $\psi:O_{\infty}\rightarrow End_X$. Recall that the extensions of this $O_{\infty}$-algebra structure into $P_{\infty}$-algebra structures are parametrized by the homotopy fiber
\[
P_{\infty}\{X,\psi\} \rightarrow P_{\infty}\{X\}\stackrel{i^*}{\rightarrow}O_{\infty}\{X\},
\]
where $i^*$ is the map induced by precomposing with $i$ and the fiber is taken over the base point $\psi$.
The map $i$ admits a factorization $O_{\infty}\stackrel{\tilde{i}}{\rightarrowtail}
\tilde{P_{\infty}}\stackrel{\sim}{\twoheadrightarrow_p}P_{\infty}$ into a cofibration $\tilde{i}$
followed by an acyclic fibration $p$. We know that $i$ induces a map of simplicial sets
$i^*:P_{\infty}\{X\}\rightarrow O_{\infty}\{X\}$ which factors through
$P_{\infty}\{X\}\stackrel{\sim}{\rightarrow_{p^*}}\tilde{P_{\infty}}\{X\}\stackrel{\tilde{i}^*}{\twoheadrightarrow}
O_{\infty}\{X\}$, and that there is a homotopy equivalence
\[
P_{\infty}\{X,\psi\} = hofib(i^*)\sim fib(\tilde{i}^*)
\]
so that the homotopy type of the relative moduli space $P_{\infty}\{X,\psi\}$ is given by the strict fiber of the Kan fibration $\tilde{i}^*$. Let us also note that, according to \cite{Yal1}, the map $p$ induces a weak equivalence
\[
\mathcal{N}p^*:\mathcal{N}wCh_{\mathbb{K}}^{P_{\infty}}\rightarrow \mathcal{N}wCh_{\mathbb{K}}^{\tilde{P_{\infty}}}.
\]

The main purpose of this section is to prove the following results, which are adaptations to the relative case
of the results obtained in the previous sections:
\begin{thm}
Let $X$ be a dg $O_{\infty}$-algebra.
Then the commutative square
\[
\xymatrix{P_{\infty}\{X,\psi\}\ar[d]\ar[r] & \mathcal{N}wCh_{\mathbb{K}}^{P_{\infty}}\ar[d]^{\mathcal{N}i^*}\\
\{X\}\ar[r] & \mathcal{N}wCh_{\mathbb{K}}^{O_{\infty}}
}
\]
is a homotopy pullback of simplicial sets.
\end{thm}

\begin{thm}
Let us suppose that $X$ is a graded $O_{\infty}$-algebra with a trivial differential (e.g $X$ is the cohomology of some cochain complex). Let $\phi$ be a $P_{\infty}$-algebra structure on $X$ extending its $O_{\infty}$-algebra structure.
There exists a commutative square
\[
\xymatrix{
WL^HwCh_{\mathbb{K}}^{O_{\infty}}(X,X)\times_{L^HwCh_{\mathbb{K}}^{O_{\infty}}(X,X)} P_{\infty}\{X,\psi\}_{[\phi]}^f \ar[r]^-{\sim} \ar@{->>}[d]_{\pi}
& diag\mathcal{N} fwCh_{\mathbb{K}}^{P\otimes\Delta^{\bullet}}|_X\ar[d]^{diag\mathcal{N}(i\otimes\Delta^{\bullet})^*}\\
\overline{W}L^HwCh_{\mathbb{K}}^{O_{\infty}}(X,X) \ar[r]^-{\sim} & \mathcal{N}wCh_{\mathbb{K}}^{O_{\infty}}|_X
}
\]
where $\pi$ is a Kan fibration obtained by the simplicial Borel construction
and the horizontal maps are weak equivalences of simplicial sets.
\end{thm}

\begin{cor}
There is an injection
\[
Aut_{Ho(Ch_{\mathbb{K}}^{O_{\infty}})}(X,\psi)/
Aut_{Ho(Ch_{\mathbb{K}}^{P_{\infty}})}(X,\phi)\hookrightarrow\pi_0P_{\infty}\{X,\psi\}_{[\phi]}
\]
\end{cor}
Here
\[Aut_{Ho(Ch_{\mathbb{K}}^{P_{\infty}})}(X,\phi)
\]
is seen as the subgroup of automorphisms of
\[
Aut_{Ho(Ch_{\mathbb{K}}^{O_{\infty}})}(X,\psi)
\]
which realizes as automorphisms of $X$
in the category of $P_{\infty}$-algebras.
The higher homotopy groups are more difficult to compute in the relative case, since a priori the
$\pi_n\overline{W}L^HwCh_{\mathbb{K}}^{O_{\infty}}(X)$ are non zero for $n\geq 2$.
However, \cite[Corollary 2.18]{Yal4} provides a cohomological interpretation of the long
exact sequence.
\begin{proof}[Proof of Theorem 4.2]
We have a commutative diagram
\[
\xymatrix{holim\ar[r]\ar[d] & diag\mathcal{N} fwCh_{\mathbb{K}}^{P_{\infty}\otimes\Delta^{\bullet}}\ar[d]\ar[r]^{\sim}
& diag\mathcal{N} wCh_{\mathbb{K}}^{P_{\infty}\otimes\Delta^{\bullet}}\ar[d]
& \mathcal{N} wCh_{\mathbb{K}}^{P_{\infty}}\ar[l]_-{\sim}\ar[d]^{\mathcal{N}i^*}\\
pt\ar[r] & diag\mathcal{N}fwCh_{\mathbb{K}}^{O_{\infty}\otimes\Delta^{\bullet}}\ar[r]^{\sim} &
diag\mathcal{N}wCh_{\mathbb{K}}^{O_{\infty}\otimes\Delta^{\bullet}}  & \mathcal{N}wCh_{\mathbb{K}}^{O_{\infty}}
\ar[l]_-{\sim}}.
\]
We want to prove that $holim$ (the homotopy pullback of the left hand commutative square) has the homotopy type
of $P_{\infty}\{X,\psi\}_{[\phi]}$. For this, we will prove that $holim$ has the homotopy of $fib(\tilde{i^*})$, the fiber
of $\tilde{i^*}$. We follow arguments similar to those of Rezk in \cite{Rez}.

We recall the variant of Quillen's Theorem B used in \cite{Rez}. Let $M^{\bullet}$ be a simplicial category, $N$
a category and $\pi:M^{\bullet}\rightarrow cs_{\bullet}N$ a simplicial functor where $cs_{\bullet}N$
is the constant simplicial category over $N$. Then the commutative square
\[
\xymatrix{
diag\mathcal{N}(\pi\downarrow X)\ar[r]\ar[d] & diag\mathcal{N}M^{\bullet}\ar[d]^{diag\pi} \\
\mathcal{N}(N\downarrow X)\ar[r] & \mathcal{N}N }
\]
is a homotopy pullback if each $diag\mathcal{N}(\pi\downarrow X)\stackrel{\sim}{\rightarrow}
diag\mathcal{N}(\pi\downarrow X')$ induced by a morphism $X\rightarrow X'$ of $N$
is a weak equivalence.
Here we consider the functor
\[
\pi:fwCh_{\mathbb{K}}^{P_{\infty}\otimes\Delta^{\bullet}}\rightarrow cs_{\bullet}fwCh_{\mathbb{K}}^{P_{\infty}}
\stackrel{cs_{\bullet}i^*}{\rightarrow} cs_{\bullet}fwCh_{\mathbb{K}}^{O_{\infty}}
\]
induced by
\[
O_{\infty}\stackrel{i}P_{\infty}\cong P_{\infty}\otimes\Delta^0\hookrightarrow P_{\infty}\otimes\Delta^{\bullet}.
\]
Recall that we consider a cosimplicial frame $(-)\otimes\Delta^{\bullet}$ on $P_{\infty}$,
which is sufficient to get a cosimplicial resolution of $P_{\infty}$ since $P_{\infty}$
is cofibrant (see \cite[Chapter 16]{Hir}).
We note $\mathcal{N}(\pi\downarrow X)_{\bullet,\bullet}$ in order to distinguish between the two simplicial degrees,
the first $\bullet$ being the simplicial dimension of the nerve and the second $\bullet$ being
the simplicial dimension of the simplicial category $(\pi\downarrow X)$.
The key point is to establish, for every $s$, the homotopy equivalence
\[
\mathcal{N}(\pi\downarrow X)_{s,\bullet}\sim\coprod_{Y_s\stackrel{\sim}{\twoheadrightarrow}...\stackrel{\sim}{\twoheadrightarrow}
Y_0\rightarrow X\in fwCh_{\mathbb{K}}^{O_{\infty}}}P_{\infty}\{Y_s\stackrel{\sim}{\twoheadrightarrow}...\stackrel{\sim}{\twoheadrightarrow}
Y_0,\psi_{\{Y_i\}}\},
\]
where $\psi_{\{Y_i\}}$ is the $O_{\infty}$-algebra structure on the sequence of arrows $Y_s\stackrel{\sim}{\twoheadrightarrow}...\stackrel{\sim}{\twoheadrightarrow} Y_0$.
Once this homotopy equivalence holds, the remaining part of the proof is exactly the argument of the proof of \cite{Rez}.
We have
\begin{eqnarray*}
\mathcal{N}(\pi\downarrow X)_{s,\bullet} & = & \{\pi(Y_s)\stackrel{\sim}{\twoheadrightarrow_{\pi(f_s)}}...\stackrel{\sim}{\twoheadrightarrow}
\pi(Y_0)\rightarrow X \}\\
 & = &  \{ \pi(Y_s\stackrel{\sim}{\twoheadrightarrow}...\stackrel{\sim}{\twoheadrightarrow}Y_0)\rightarrow X
 \in fwCh_{\mathbb{K}}^{O_{\infty}},
 Y_s\stackrel{\sim}{\twoheadrightarrow}...\stackrel{\sim}{\twoheadrightarrow}Y_0
 \in (\mathcal{N}fwCh_{\mathbb{K}}^{P_{\infty}\otimes\Delta^{\bullet}})_{s,\bullet} \} \\
 & = & \coprod_{Y_s\stackrel{\sim}{\twoheadrightarrow}...\stackrel{\sim}{\twoheadrightarrow}
Y_0\rightarrow X\in fwCh_{\mathbb{K}}^{O_{\infty}}} fib(\tilde{P_{\infty}}\{Y_s\stackrel{\sim}{\twoheadrightarrow}...\stackrel{\sim}{\twoheadrightarrow}
Y_0 \}
\stackrel{\tilde{i}^*}{\twoheadrightarrow} O_{\infty}\{Y_s\stackrel{\sim}{\twoheadrightarrow}...\stackrel{\sim}{\twoheadrightarrow}
Y_0 \})\\
 & \sim & \coprod_{Y_s\stackrel{\sim}{\twoheadrightarrow}...\stackrel{\sim}{\twoheadrightarrow}
Y_0\rightarrow X\in fwCh_{\mathbb{K}}^{O_{\infty}}}P_{\infty}\{Y_s\stackrel{\sim}{\twoheadrightarrow}...\stackrel{\sim}{\twoheadrightarrow}
Y_0,\psi_{\{Y_i\}}\} \\
\end{eqnarray*}
where, in the third line, the basepoint of each $O_{\infty}\{Y_s\stackrel{\sim}{\twoheadrightarrow}...\stackrel{\sim}{\twoheadrightarrow}
Y_0\}$ is given by the $Y_s\stackrel{\sim}{\twoheadrightarrow}...\stackrel{\sim}{\twoheadrightarrow}
Y_0\in fwCh_{\mathbb{K}}^{P_{\infty}}$ indexing the coproduct.
\end{proof}

The homotopy fiber
\[
\xymatrix{P_{\infty}\{X,\psi\}\ar[r]\ar[d] & diag\mathcal{N} fwCh_{\mathbb{K}}^{P_{\infty}\otimes\Delta^{\bullet}}\ar[d]\\
\{X\}\ar[r] & diag\mathcal{N}fwCh_{\mathbb{K}}^{O_{\infty}\otimes\Delta^{\bullet}}
\sim  \mathcal{N}wCh_{\mathbb{K}}^{O_{\infty}}
}
\]
restricts to a homotopy fiber
\[
\xymatrix{P_{\infty}\{X,\psi\}_{[\phi]}\ar[r]\ar[d] & diag\mathcal{N} fwCh_{\mathbb{K}}^{P\otimes\Delta^{\bullet}}|_X\ar[d]\\
\{X,\psi\}\ar[r] &  \mathcal{N}wCh_{\mathbb{K}}^{O_{\infty}}|_{(X,\psi)}
},
\]
where $\mathcal{N}wCh_{\mathbb{K}}^{O_{\infty}}|_{(X,\psi)}$  is the connected component
of the $O_{\infty}$-algebra $(X,\psi)$, and
\[
diag\mathcal{N} fwCh_{\mathbb{K}}^{P\otimes\Delta^{\bullet}}|_X
=\coprod_{[X,\varphi]} diag\mathcal{N} fwCh_{\mathbb{K}}^{P\otimes\Delta^{\bullet}}|_{(X,\varphi)}
\]
is the union of the connected components $diag\mathcal{N} fwCh_{\mathbb{K}}^{P\otimes\Delta^{\bullet}}|_{(X,\varphi)}$
of the $(X,\varphi)$ ranging over the acyclic fibration classes of $P_{\infty}$-algebras
having $X$ as underlying complex, and such that there exists a zigzag
$(X,\varphi\circ i)\stackrel{\sim}{\leftarrow}\bullet\stackrel{\sim}{\rightarrow}(X,\psi)$
of weak equivalences of $O_{\infty}$-algebras.
The Kan subcomplex $P_{\infty}\{X,\psi\}_{[\phi]}$ of $P_{\infty}\{X,\psi\}$ is defined in the same way as in the non relative setting, by requiring moreover that the $P_{\infty}$-algebra structures on $X$ extend $\psi$. Actually one can see that it is also given by the homotopy fiber of the map $P_{\infty}\{X\}_{[\phi]}\rightarrow O_{\infty}\{X\}_{[\psi]}$ induced by $i$.
Let us note that we use here the identification of $P_{\infty}\{X,\psi\}_{[\phi]}$ with $P_{\infty}\{X,\psi\}_{[\phi]}^f$ (definition analogous to the non relative case), which follows from a relative version of Lemma 2.7 obtained by applying the lifting arguments of the proof of Lemma 2.7 in the comma category $(\mathcal{P} \downarrow O_{\infty})$.

Now let us prove Theorem 4.3:
\begin{proof}[Proof of Theorem 4.3]
Recall that the hammocks of weight $k$ and length $n$, which define the $k$-simplices of the
simplicial monoid $L^HwCh_{\mathbb{K}}^{P_{\infty}}(X,X)$, consist in the data of $k-1$ chains of morphisms of length $n$
\[
X\stackrel{f_0^1}{\leftarrow}\bullet...\bullet\stackrel{f_0^n}{\rightarrow}X,
\]
...,
\[
X\stackrel{f_k^1}{\leftarrow}\bullet...\bullet\stackrel{f_k^n}{\rightarrow}X
\]
organized in a commutative diagram (a ``hammock")
\[
\xymatrix{
 & C_{0,1}\ar[dl]_{f_0^1}\ar[d]^{...}\ar[r] & C_{0,2}\ar[d]^{...} &...\ar[l] &C_{0,n}\ar[l]_{f_0^{n-1}}\ar[d]^{...} \ar[dr]^{f_0^n} & \\
X & ... \ar[l]\ar[d]^{...}& ... \ar[d]^{...} & ...\ar[l] & ...\ar[d]^{...}\ar[r] & X \\
& C_{k,1}\ar[ul]^{f_n^1}\ar[r] & C_{k,2} & ...\ar[l] &C_{k,n}\ar[l]_{f_k^{n-1}} \ar[ur]_{f_0^n} & \\
}
\]
We define a simplicial action of the simplicial monoid $L^HwCh_{\mathbb{K}}^{P_{\infty}}(X,X)$
on the moduli space $P_{\infty}\{X,\psi\}$ in each simplicial dimension by a map
\[
L^HwCh_{\mathbb{K}}^{P_{\infty}}(X,X)_k\times P_{\infty}\{X,\psi\}_k\rightarrow P_{\infty}\{X,\psi\}_k
\]
which associates to any pair $(\{X\stackrel{f_i^1}{\leftarrow}\bullet...\bullet\stackrel{f_i^n}{\rightarrow}X\}_{1\leq i\leq n},
\varphi:P_{\infty}\otimes\Delta^k\rightarrow End_X)$ the map $\varphi_{\overline{H^*}(f_0^1,...,f_0^n)}$
defined by the conjugation action of $\overline{H^*}(f_0^1,...,f_0^n)$ on $\phi$ like in Lemma 2.10,
with $\overline{H^*}(f_0^1,...,f_0^n)$ the composite of the homologies of each $f_0^i$ going in the right direction
and the inverse of the homologies of the each $f_0^i$ going in the left direction.
We have already seen that such a map gives a well defined monoid action.
Moreover, this conjugation action is compatible with the faces and degeneracies of $(P,O)_{\infty}\{X\}$,
and the fact that $\overline{H^*}(f_i^1,...,f_i^n)=\overline{H^*}(f_j^1,...,f_j^n)$ for every
$1\leq i,j\leq n$ (by commutativity of the hammock diagram) implies the compatibility of this action with
the faces of hammocks (vertical contraction of two chains) and their degeneracies (vertical insertion of identity
maps).

We thus obtain a simplicial action of $L^HwCh_{\mathbb{K}}^{P_{\infty}}(X,X)$ on $P_{\infty}\{X,\psi\}$.
It remains to prove that this action restrict to $P_{\infty}\{X,\psi\}_{[\phi]}$.
By the inductive construction of the simplices of the local realization space,
it is sufficient to prove it for the vertices. We apply exactly the same argument as in the proof of
Lemma 2.10.

We finally get a Kan fibration
\[
\pi:WL^HwCh_{\mathbb{K}}^{O_{\infty}}(X,\psi)\times_{L^HwCh_{\mathbb{K}}^{O_{\infty}}(X,\psi)}
(P,O)_{\infty}\{X\}_{[\phi]}\twoheadrightarrow \overline{W}L^HwCh_{\mathbb{K}}^{O_{\infty}}(X,\psi)
\]
with fiber $(P,O)_{\infty}\{X\}_{[\phi]}$.
We then construct a commutative square analogous to the square in the proof of Theorem 2.8. The only
modification is to apply $\overline{H^*}(-)$ in the definition of the upper horizontal arrow
before transferring the $P_{\infty}\otimes\Delta^k$-algebra structure along the chain of isomorphisms.
The lower horizontal arrow is a weak equivalence according to the work of Dwyer-Kan \cite{DK3},
and for the upper one we use the five lemma.
\end{proof}

The proof of Corollary 4.4 is the same as in the non relative case:
\begin{proof}
We use the long exact sequence associated to $\pi$ to get an injection
\[
\partial_1:\pi_1\overline{W}L^HwCh_{\mathbb{K}}^{O_{\infty}}(X,\psi)/Im(\pi_1(\pi))\hookrightarrow\pi_0P_{\infty}\{X,\psi\}_{[\phi]}.
\]
The isomorphisms
\begin{eqnarray*}
\pi_1\overline{W}L^HwCh_{\mathbb{K}}^{O_{\infty}}(X) & \cong & \pi_0L^HwCh_{\mathbb{K}}^{O_{\infty}}(X)\\
 & = & Aut_{Ho(Ch_{\mathbb{K}}^{O_{\infty}})}(X)
\end{eqnarray*}
and
\begin{eqnarray*}
 & \pi_1(WL^HwCh_{\mathbb{K}}^{O_{\infty}}(X,\psi)\times_{L^HwCh_{\mathbb{K}}^{O_{\infty}}(X,\psi)}
P_{\infty}\{X,\psi\}_{[\phi]},(id,\phi)) & \\
 \cong & \pi_1\overline{W}L^HwCh_{\mathbb{K}}^{P_{\infty}}(X) & \\
 = & Aut_{Ho(Ch_{\mathbb{K}}^{P_{\infty}})}(X) & \\
\end{eqnarray*}
(which follows from the equivalence of fiber sequences)
allow us to conclude.
\end{proof}

\begin{example}
When $Aut_{Ho(Ch_{\mathbb{K}}^{O_{\infty}})}(X)\cong Aut_O(X)$ (e.g for the trivial $O_{\infty}$-algebra structure)
and $Aut_P(X)=Aut_O(X)$, then $\pi_0P_{\infty}\{X,\psi\}_{[\phi]} = *$. We will see an important example of this condition with the realization of Poincaré duality
up to a fixed rational homotopy type.
\end{example}

In the operadic case, we can compare by applying the obstructions criteria of \cite{Hof}.
A typical situation is when $O$ is an operad which embeds into a larger algebraic structure encoded by a prop $P$.
When $H\Gamma_O^0(X,X)=0$ and $Aut_O(X)=Aut_P(X)$, then $\pi_0P_{\infty}\{X,\psi\}_{[\phi]} = *$.

\section{Realization of Poincaré Duality on cochains}

Poincaré duality can be realized at the level of cochains
in a Frobenius algebra up to homotopy, i.e a homotopy Frobenius algebra, which is expected to contain in its higher operations some information of geometrical nature
when one considers an oriented compact manifold.
However, no one knows how many structures (up to homotopy) exist on the cochains, e.g. if there are non equivalent realizations
of Poincaré duality. In particular, how many structures extend the usual $E_{\infty}$ structure on cochains which fixes
the rational homotopy type.
We provide here concrete computations of connected components of realizations spaces for a broad range of examples of
formal manifolds in the sense of rational homotopy theory \cite{DGMS}.

\subsection{Frobenius algebras and their variants}

We recall some definitions of dg Frobenius algebras and bialgebras, and point out how one
can pass from algebras to bialgebras.
We are particulary interested in special symmetric Frobenius algebras, which naturally appear
on the cohomology of Poincaré duality spaces, and play also a key role in the construction of
the algebraic counterpart of conformal field theories (in the appropriate tensor category, see \cite{FRS2}).

\begin{defn}
(1) A (differential graded) Frobenius algebra is a unitary dg associative algebra of
finite dimension $A$ endowed with a non-degenerate bilinear form $<.,.>:A\otimes A\rightarrow \mathbb{K}$
which is invariant with respect to the product, i.e $<xy,z>=<x,yz>$. Equivalently, a Frobenius algebra is a unitary dg associative algebra of finite dimension $A$ with product $\mu$, equipped moreover with a linear form $\epsilon:A\rightarrow\mathbb{K}$ such that $\epsilon\circ\mu$ is a non-degenerate bilinear form.

(2) A Frobenius algebra $A$ is symmetric if this bilinear form is symmetric.

(3) A Frobenius algebra is commutative if its product is commutative.
\end{defn}
Let us note that a commutative Frobenius algebra is symmetric, but the converse is not true.
A topological instance of symmetric Frobenius algebra is the cohomology ring (over a field) of
a Poincaré duality space.
There is also a notion of Frobenius bialgebra:
\begin{defn}
(1) A differential graded Frobenius bialgebra of degree $m$ is a triple $(B,\mu,\Delta)$ such that:

(i) $(B,\mu)$ is a dg associative algebra;

(ii) $(B,\Delta)$ is a dg coassociative coalgebra with $deg(\Delta)=m$;

(iii) the map $\Delta:B\rightarrow B\otimes B$ is a morphism of left $B$-module
and right $B$-module, i.e in Sweedler's notations we get the equalities
\begin{eqnarray*}
\sum_{(x.y)}(x.y)_{(1)}\otimes (x.y)_{(2)} & = & \sum_{(y)}x.y_{(1)}\otimes y_{(2)}\\
 & = & \sum_{(x)}(-1)^{m|x|}x_{(1)}\otimes x_{(2)}.y
\end{eqnarray*}
called the Frobenius relations.

(2) A Frobenius bialgebra is commutative if the product is commutative, and cocommutative if the coproduct is cocommutative.

(3) A Frobenius bialgebra is symmetric if it is counitary, and the composite of the product with the counit is a symmetric bilinear form.
\end{defn}
Now we give a presentation of the properad parameterizing such structures in terms of directed graphs
with a flow going from the top to the bottom.
The properad $Frob^m$ of Frobenius bialgebras of degree $m$ is generated by an operation of arity $(2,1)$ and of degree $0$
\[
\xymatrix @R=0.5em@C=0.75em@M=0em{
\ar@{-}[dr] & & \ar@{-}[dl] \\
 & \ar@{-}[d] &  \\
 & & }
\]
and an operation of arity $(1,2)$ and of degree $m$
\[
\xymatrix @R=0.5em@C=0.75em@M=0em{
 & \ar@{-}[d] & \\
 & \ar@{-}[dl] \ar@{-}[dr] & \\
 & & }
\]
which are invariant under the action of $\Sigma_2$.
It is quotiented by the ideal generated by the following relations:
\begin{itemize}
\item[]\textbf{Associativity and coassociativity}
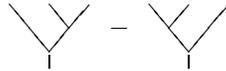
\[
\xymatrix @R=0.5em@C=0.75em@M=0em{
 \ar@{-}[dr] & & \ar@{-}[dr] & & \ar@{-}[dl] & &  \ar@{-}[dr] & & \ar@{-}[dl] & & \ar@{-}[dl]\\
 & \ar@{-}[dr] & &\ar@{-}[dl] & & - &  & \ar@{-}[dr] & &\ar@{-}[dl] & \\
 & & \ar@{-}[d] & & & & & & \ar@{-}[d] & & & \\
 & & & & & & & & & & \\
}
\]
and
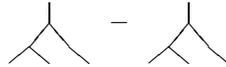
\[
\xymatrix @R=0.5em@C=0.75em@M=0em{
 & & \ar@{-}[d] & & & & & & \ar@{-}[d] & & & \\
 & &\ar@{-}[dl] \ar@{-}[dr]& & & - &  & &\ar@{-}[dl] \ar@{-}[dr] & & \\
 &\ar@{-}[dl] \ar@{-}[dr] & &\ar@{-}[dr] & & & & \ar@{-}[dr] \ar@{-}[dl]& & \ar@{-}[dr] & \\
 & & & & & & & & & & \\
}
\]

\item[]\textbf{Frobenius relations}
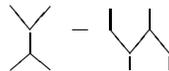
\[
\xymatrix @R=0.5em@C=0.75em@M=0em{
\ar@{-}[dr] & & \ar@{-}[dl] & & \ar@{-}[d] & & \ar@{-}[d] & \\
 & \ar@{-}[d] & & - & \ar@{-}[dr] & & \ar@{-}[dl] \ar@{-}[dr] & \\
 & \ar@{-}[dl] \ar@{-}[dr] & & & & \ar@{-}[d] & & \ar@{-}[d] \\
 & & & & & & & \\
}
\]
and
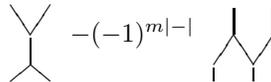
\[
\xymatrix @R=0.5em@C=0.75em@M=0em{
\ar@{-}[dr] & & \ar@{-}[dl] & & &\ar@{-}[d] & & \ar@{-}[d] \\
 & \ar@{-}[d] & & - (-1)^{m|-|} & & \ar@{-}[dl] \ar@{-}[dr]& & \ar@{-}[dl] \\
 & \ar@{-}[dl] \ar@{-}[dr] & & & \ar@{-}[d] & & \ar@{-}[d] & \\
 & & & & & & & \\
}
\]
\end{itemize}
In the unitary and counitary case, one adds a generator for the unit, a generator for the counit and the necessary
compatibility relations with the product and the coproduct. We note the corresponding properad $ucFrob^m$.
We refer the reader to \cite{Koc} for a detailed survey about the role of these operations
and relations in the classification of two-dimensional topological quantum field theories.

Definitions 5.1 and 5.2 are strongly related. Indeed, if $A$ is a Frobenius algebra, then the pairing
$<.,.>$ induces an isomorphism of $A$-modules $A\cong A^*$, hence a map
\[
\Delta:A\stackrel{\cong}{\rightarrow} A^* \stackrel{\mu^*}{\rightarrow} (A\otimes A)^*\cong A^*\otimes
A^*\cong A\otimes A
\]
which equips $A$ with a structure of Frobenius bialgebra.
Conversely, one can prove that every unitary counitary Frobenius bialgebra gives rise to a Frobenius algebra,
which are finally two equivalent notions.
Let $B$ be a unitary counitary Frobenius bialgebra, with a product
\begin{tikzpicture}[scale=0.2]
\draw (0,0) -- (1,-1) -- (2,0);
\draw (1,-1) -- (1,-2);
\end{tikzpicture},
a coproduct
\begin{tikzpicture}[scale=0.2]
\draw (0,0) -- (1,1) -- (2,0);
\draw (1,1) -- (1,2);
\end{tikzpicture},
a unit
\begin{tikzpicture}[scale=0.2]
\node (X) at (0,0) {$\bullet$};
\draw (0,0) -- (0,-1);
\end{tikzpicture}
and a counit
\begin{tikzpicture}[scale=0.2]
\node (X) at (0,0) {$\bullet$};
\draw (0,0) -- (0,1);
\end{tikzpicture}.
The composite
\begin{tikzpicture}[scale=0.2]
\draw (0,0) -- (1,-1) -- (2,0);
\draw (1,-1) -- (1,-2);
\node (X) at (1,-2) {$\bullet$};
\end{tikzpicture}
is a bilinear form, invariant with respect to the product by associativity:
\[
\begin{tikzpicture}[scale=0.3]
\draw (0,0) -- (2,-2);
\draw (1,-1) -- (2,0);
\draw (2,-2) -- (4,0);
\draw (2,-2) -- (2,-3);
\node (B) at (2,-3) {$\bullet$};
\node (A) at (5,-2) {=};
\draw (6,0) -- (8,-2);
\draw (9,-1) -- (8,0);
\draw (8,-2) -- (10,0);
\draw (8,-2) -- (8,-3);
\node (C) at (8,-3) {$\bullet$};
\end{tikzpicture}.
\]
Let us draw the Frobenius relation:
\[
\begin{tikzpicture}[scale=0.2]
\draw (0,0) -- (1,-1) -- (2,0);
\draw (1,-1) -- (1,-2);
\draw (0,-3) -- (1,-2) -- (2,-3);
\node (A) at (3,-2) {$=$};
\draw (4,-3) -- (4,-2) -- (5,-1) -- (6,-2) -- (7,-1) -- (7,0);
\draw (5,0) -- (5,-1);
\draw (6,-2) -- (6,-3);
\node (B) at (8,-2) {$=$};
\draw (9,0) -- (9,-1) -- (10,-2) -- (11,-1) -- (12,-2) -- (12,-3);
\draw (11,0) -- (11,-1);
\draw (10,-2) -- (10,-3);
\end{tikzpicture}.
\]
The composite
\begin{tikzpicture}[scale=0.2]
\draw (0,0) -- (1,1) -- (2,0);
\draw (1,1) -- (1,2);
\node (X) at (1,2) {$\bullet$};
\end{tikzpicture}
defines a copairing, which is actually the dual copairing of
\begin{tikzpicture}[scale=0.2]
\draw (0,0) -- (1,-1) -- (2,0);
\draw (1,-1) -- (1,-2);
\node (X) at (1,-2) {$\bullet$};
\end{tikzpicture}
according to Frobenius relation. In particular, this bilinear form is non-degenerate
and we finally obtain the desired Frobenius algebra structure.
We refer the reader to \cite{Koc} for a detailed survey about the role of these operations
and relations in the classification of two-dimensional topological quantum field theories.
\begin{rem}
Another way to encode Frobenius algebras is to define them as algebras over the commutative cyclic operad.
However, there is no known model category structure on dg cyclic operads (and explicit constructions
of resolutions), so the equivalent properadic notion
of unitary counitary Frobenius bialgebra is preferable for our purposes.
\end{rem}
Finally let us define a particular sort of Frobenius algebras:
\begin{defn}
A special Frobenius algebra is a unitary and counitary Frobenius bialgebra $(B,\mu,\eta,\epsilon)$,
where $\eta$ denotes the unit and $\epsilon$ the counit, such that $\epsilon\circ\eta=\lambda_0.id_{\mathbb{K}}$
and $\mu\circ\Delta=\lambda_1.id_B$ for fixed $\lambda_0,\lambda_1\in\mathbb{K}$.
In terms of graphs it gives
\[
\begin{tikzpicture}[scale=0.2]
\draw (0,0) -- (0,-2);
\node (A) at (0,0) {$\bullet$};
\node (B) at (0,-2) {$\bullet$};
\node (C) at (2,-1) {$= \lambda_0$};
\end{tikzpicture}
\]
and
\[
\begin{tikzpicture}[scale=0.2]
\draw (0,0) -- (1,1) -- (2,0);
\draw (1,1) -- (1,2);
\draw (0,0) -- (1,-1) -- (2,0);
\draw (1,-1) -- (1,-2);
\node (A) at (4,-1) {$= \lambda_1.$};
\draw (6,0) -- (6,-2);
\end{tikzpicture}.
\]
It is said to be normalized when $\lambda_0=dim(B)$ and $\lambda_1=1$.
\end{defn}
We denote by $sFrob$ the properad encoding special Frobenius bialgebras.

\begin{example}
Let $M$ be an oriented connected closed manifold of dimension $n$. Let $[M]\in H_n(M;\mathbb{K})\cong H^0(M;\mathbb{K})\cong \mathbb{K}$
be the fundamental class of $[M]$. Then the cohomology ring $H^*(M;\mathbb{K})$ of $M$ inherits a structure of commutative
and cocommutative Frobenius bialgebra of degree $n$ with the following data:

(i) the product is the cup product
\begin{eqnarray*}
\mu: H^kM\otimes H^lM & \rightarrow & H^{k+l}M \\
 x\otimes y & \mapsto & x\cup y
\end{eqnarray*}

(ii) the unit $\eta:\mathbb{K}\rightarrow H^0M\cong H_nM$ sends $1_{\mathbb{K}}$ to the fundamental class $[M]$;

(iii) the non-degenerate pairing is given by the Poincaré duality:
\begin{eqnarray*}
\beta: H^kM\otimes H^{n-k}M & \rightarrow & \mathbb{K}\\
 x\otimes y & \mapsto & <x\cup y,[M]>
\end{eqnarray*}
i.e the evaluation of the cup product on the fundamental class;

(iv) the coproduct $\Delta=(\mu\otimes id)\circ (id\otimes \gamma)$ where
\[
\gamma: \mathbb{K}\rightarrow \bigoplus_{k+l=n}H^kM\otimes H^lM
\]
is the dual copairing of $\beta$, which exists since $\beta$ is non-degenerate;

(v) the counit $\epsilon=<.,[M]>:H^nM\rightarrow \mathbb{K}$ i.e the evaluation on the fundamental class.
\end{example}
The proposition below sums up classical properties of symmetric Frobenius bialgebras:
\begin{prop}
(1) A Frobenius algebra is symmetric and finite dimensional.

(2) A counitary Frobenius bialgebra is a Frobenius algebra, in particular it is finite dimensional.

(3) A symmetric Frobenius bialgebra which is commutative is also cocommutative.

(4) For every element $z$ of a symmetric Frobenius bialgebra $A$, the element $\sum_{(z)}z_{(2)}z_{(1)}$
lies in the center of $A$.
\end{prop}

\subsection{Computations of connected components}

Special symmetric Frobenius bialgebras satisfy a strong rigidity property given by the following result of \cite{FRS}:
\begin{thm}(\cite[Theorem 3.6]{FRS})
If an algebra $A$ can be endowed with a special symmetric Frobenius bialgebra structure,
then this structure is the unique structure extending the algebra structure, up to normalization of the counit.
\end{thm}
This result is available in a wide range of tensor categories, in particular cochain complexes over a field.
Let us denote by $sFrob$ the properad encoding special commutative and cocommutative Frobenius bialgebras
(which are in particular symmetric). Let us suppose that $H^*X$ is a commutative algebra induced by a homotopy commutative structure $\psi:Com_{\infty}\rightarrow End_X$ (for instance, the one induced by higher cup products on the singular cohomology of a topological space). We define the relative realization space $Real_{sFrob_{\infty}}\{H^*X,\psi\}$ as the Kan subcomplex of $Real_{sFrob_{\infty}}\{H^*X\}$ defined by the realizations of the special Frobenius algebra structure on $H^*X$ which extend $\psi$.
\begin{cor}
If there exists $sFrob\rightarrow End_{H^*X}$ extending the commutative structure of $H^*X$, then
\[
sFrob_{\infty}\{X,\psi\}=Real_{sFrob_{\infty}}\{H^*X,\psi\}.
\]
\end{cor}
These algebras are used in conformal field theory.
An important example is provided by Poincaré duality:
\begin{lem}
Let $M$ be a compact oriented manifold. Then $H^*(M;\mathbb{K})$ is a special symmetric Frobenius bialgebra.
\end{lem}
\begin{proof}
We know that $(\epsilon\circ\eta)(1_{\mathbb{K}})=1_{\mathbb{K}}$ by definition, since $\epsilon$
sends $1$ on the fundamental class $[M]$ of $M$, and $\eta$ sends the fundamental class of $M$ on $1$.
Hence $\epsilon\circ\eta=id_{\mathbb{K}}$.
Let $\beta:H^*M\otimes H^*M\rightarrow\mathbb{K}$ be the non degenerate bilinear form associated to Poincaré
duality and $\gamma$ be its dual copairing. We know by construction that
\[
\gamma\circ\beta=\epsilon\circ\mu\Delta\circ\eta = \chi(M).id_{\mathbb{K}},
\]
where $\chi(M)$ is the Euler characteristic of $M$. Since $\epsilon$ sends $[M]$ on $1$, we get
\[
\mu\Delta\circ\eta=(\chi(M).id_{\mathbb{K}}).[M],
\]
in particular $(\mu\Delta\circ\eta)(1)=\chi(M).[M]$
(using also the fact that $\mu\Delta\circ\eta$ sends $\mathbb{K}$ in $H^0(M)$).
Now, using the fact that $\Delta=(\mu\otimes id)\circ(id\otimes\gamma)$, we get
\begin{eqnarray*}
\mu\circ\Delta & = & \mu\circ(\mu\otimes id)\circ(id\otimes\gamma) \\
 & = & \mu\circ(id\otimes (\mu\circ\gamma)) \\
 & = & \mu\circ(id\otimes (\chi(M).id_{\mathbb{K}}).[M])\\
 & = & (\chi(M).[M]).Id_{H^*M} \\
 & = & \chi(M).Id_{H^*M}.
\end{eqnarray*}
The first line is by construction of $\Delta$, the second line follows from the associativity of $\mu$,
and the last line because $[M]$ is the unit for the product.
Moreover, since $H^*M$ is commutative and cocommutative, it is in particular symmetric.
\end{proof}

We will also need the following rigidity result about automorphisms of special Frobenius bialgebras:
\begin{lem}(see \cite{FRS2})
Let $A$ be a special symmetric Frobenius bialgebra. A morphism in $End(A)$ is an algebra automorphism
if and only if it is a coalgebra automorphism.
\end{lem}
Maps between Frobenius bialgebras which are simultaneously morphisms of algebras and morphisms of coalgebras
are exactly the morphisms of $sFrob$-algebras, thus automorphisms of a symmetric special Frobenius bialgebra
are exactly its automorphisms as a unitary commutative algebra.
Now we use:
\begin{lem}
Every algebra automorphism of a unitary algebra preserves the unit.
\end{lem}
\begin{proof}
Let $\varphi:A\rightarrow A$ be such an automorphism. We have, for every $x\in A$,
$\varphi(x)=\varphi(1_A.x)=\varphi(1_A)\varphi(x)$.
Since $\varphi$ is bijective, for every $x\in A$ there exists $t\in A$ such that $x=\varphi(t)$,
so
\[
\varphi(1_A)x=\varphi(1_A)\varphi(t)=\varphi(1_A.t)=\varphi(t)=x,
\]
in particular
\[
\varphi(1_A)=\varphi(1_A).1_A=1_A.
\]
\end{proof}
We deduce:
\begin{lem}
For every special symmetric Frobenius bialgebra $A$, we have
\[
Aut_{sFrob}(A)=Aut_{uCom}(A)=Aut_{Com}(A).
\]
\end{lem}
Hence
\begin{prop}
We have
\[
Aut_{\mathbb{K}}(H^*X)/Aut_{Com}(H^*X)\hookrightarrow\pi_0sFrob_{\infty}\{H^*X\}_{[\phi]}
\]
where $\phi$ is the trivial $sFrob_{\infty}$-structure $sFrob_{\infty}\stackrel{\sim}{\rightarrow}sFrob
\rightarrow End_{H^*X}$.
\end{prop}
To put these observations into words, we proved that we can compute a subset of
connected components of the local realization space of Poincaré duality on cochains, around the trivial homotopy structure, as a quotient of its graded vector space automorphisms by its commutative algebra automorphisms.
We use this result to do computations for some examples of manifolds.
\begin{example}
The cohomology ring of the $n$-sphere $S^n$ is $H^*S^n=\mathbb{Q}[x]/(x^2)$. A basis is given by $1,x$, so
\[
Aut_{\mathbb{Q}}(H^*S^n)\cong\mathbb{Q}^*\oplus\mathbb{Q}^*.
\]
To determine an automorphism preserving the commutative structure, we just have to fix a non zero value on $x$, so
\[
Aut_{Com}(H^*S^n)\cong\mathbb{Q}^*.
\]
We deduce that
\[
\mathbb{Q}^*\hookrightarrow\pi_0sFrob_{\infty}\{H^*X\}_{[\phi]}.
\]
\end{example}
\begin{example}
The cohomology ring of the complex projective space $\mathbb{C}P^n$ is $H^*\mathbb{C}P^n=\mathbb{Q}[x]/(x^{n+1})$.
A basis is given by  $[\mathbb{C}P^n],x,x^2,...,x^n$, so
\[
Aut_{\mathbb{Q}}(H^*\mathbb{C}P^n)\cong (\mathbb{Q}^*)^{\oplus n+1}.
\]
To determine an automorphism preserving the commutative structure, we just have to fix a non zero value on $x$ (the unit $[\mathbb{C}P^n]$ is preserved), so
\[
Aut_{Com}(H^*\mathbb{C}P^n)\cong\mathbb{Q}^*.
\]
We deduce that
\[
(\mathbb{Q}^*)^{\oplus n}\hookrightarrow\pi_0sFrob_{\infty}\{H^*X\}_{[\phi]}.
\]
\end{example}
\begin{example}
Let $S$ be a compact oriented surface of genus $g$. The zero and second cohomology groups of $S$ are $\mathbb{Q}$, and the first cohomology group
is generated by $2g$ generators. We thus have
\[
Aut_{\mathbb{Q}}(H^*S)\cong\mathbb{Q}^*\oplus GL_{2g}(\mathbb{Q})\oplus\mathbb{Q}^*.
\]
To determine an automorphism preserving the commutative structure, we just have to fix a non zero value on each generator, so
\[
Aut_{Com}(H^*S)\cong GL_{2g}(\mathbb{Q})\oplus\mathbb{Q}^*.
\]
We deduce that
\[
\mathbb{Q}^*\hookrightarrow\pi_0sFrob_{\infty}\{H^*S\}_{[\phi]}.
\]
\end{example}
\begin{example}
One constructs the cohomology ring of a connected sum of two orientable $n$-manifolds
$M\sharp N$ as follows. Using a Mayer-Vietoris exact sequence argument, one checks that
the zeroth cohomology group of $M\sharp N$ is generated by the fundamental class $[M\sharp N]$, the $n^{th}$ group is generated by a top cohomology class and
\[
H^iM\sharp N=H^iM\oplus H^iN
\]
for $0<i<n$. The cup product of two classes belonging respectively to $H^iM$ and $H^jN$ is zero.
The cup product of two classes belonging respectively to $H^iM$ and $H^jM$ is their cup product in $H^{i+j}M$ for $i+j<n$,
and is identified with the top cohomology class of $M\sharp N$ for $i+j=n$.

Let $M=\mathbb{C}P^2 \sharp \mathbb{C}P^2$ be the connected sum of two copies of $\mathbb{C}P^2$. This is a $4$-dimensional orientable closed manifold.
The zeroth and fourth cohomology groups are isomorphic to $\mathbb{Q}$.
For $0<i<4$, we have
\[
H^iM\cong H^i(\mathbb{C}P^2)\oplus H^i(\mathbb{C}P^2),
\]
hence $H^1M=H^3M=0$ and $H^2M\cong \mathbb{Q}x_1\oplus\mathbb{Q}x_2$.
Regarding the ring structure, the classes $x_1^2$ and $x_2^2$ are identified with the top cohomology class generating $H^4M$, hence $x_1^2=x_2^2$,
and the cup product $x_1\cup x_2$ is zero.

The automorphisms of $H^*M$ as a graded vector space are
\begin{eqnarray*}
Aut_{\mathbb{Q}}(H^*M) & = & GL_1(\mathbb{Q})\times GL_2(\mathbb{Q})\times GL_1(\mathbb{Q}) \\
 & = & \mathbb{Q}^*\times GL_2(\mathbb{Q})\times\mathbb{Q}^*.
\end{eqnarray*}
The automorphisms of $H^*M$ as a graded commutative algebra fix the unit so they will be of the form
\[
Aut_{Com}(H^*M)=H\times\mathbb{Q}^*
\]
where $H$ is a certain subgroup of $GL_2(\mathbb{Q})$ that we now explicit. Let $f$ be a commutative algebra automorphism of $H^*M$, it induces an automorphism
on $H^2M$ that we still denote by $f$. This automorphism corresponds to a certain invertible matrix
\[
\begin{pmatrix}
a_1 & b_1 \\
a_2 & b_2
\end{pmatrix}
\]
where $f(x_1)=a_1x_1+a_2x_2$ and $f(x_2)=b_1x_1+b_2x_2$.
The relation $x_1^2=x_2^2$ implies that $f(x_1)^2=f(x_2)^2$. We have
\begin{eqnarray*}
f(x_1)^2 & = & (a_1x_1+a_2x_2)^2 \\
 & = & a_1^2x_1^2+2a_1a_2x_1\cup x_2+a_2^2x_2^2 \\
 & = & (a_1^2+a_2^2)x_1^2,
\end{eqnarray*}
and
\[
f(x_2)^2=(b_1^2+b_2^2)x_1^2
\]
so
\[
a_1^2+a_2^2 = b_1^2+b_2^2,
\]
that is, the two vectors of $\mathbb{Q}^2$ defining the columns of matrix have the same norm.
The relation $x_1\cup x_2=0$ implies that $f(x_1)\cup f(x_2)=0$, that is
\begin{eqnarray*}
(a_1x_1+a_2x_2)\cup (b_1x_1+b_2x_2) & = & a_1b_1x_1^2+(a_1b_2+a_2b_1)x_1\cup x_2+a_2b_2x_2^2 \\
 & = & (a_1b_1+a_2b_2)x_1^2 \\
 & = & 0
\end{eqnarray*}
hence
\[
a_1b_1+a_2b_2=0.
\]
This means that the two column vectors of the matrix are orthogonal (their scalar product is zero).
This matrix is consequently orthogonal up to normalization, and we get a semi-direct product
\[
H=\mathbb{Q}^*\ltimes O_2(\mathbb{Q}).
\]
If we consider the cohomology with real coefficients instead of the rational one, we get the conformal orthogonal group $CO_2(\mathbb{R})$
(the group of conformal linear maps on $\mathbb{R}^2$).
We still denote by $CO_2(\mathbb{Q})$ its rational version. We have
\[
GL_2(\mathbb{Q})/CO_2(\mathbb{Q})\times\mathbb{Q}^*\hookrightarrow\pi_0sFrob_{\infty}\{H^*(\mathbb{C}P^2\sharp\mathbb{C}P^2)\}_{[\phi]}.
\]
When working over $\mathbb{R}$, by the exact sequence
\[
1\rightarrow SL_2(\mathbb{R})\rightarrow GL_2(\mathbb{R})\stackrel{det}{\rightarrow}\mathbb{R}^*\rightarrow 1,
\]
the general linear group is a semi-direct product
\[
GL_2(\mathbb{R})=\mathbb{R}^*\ltimes SL_2(\mathbb{R}).
\]
Moreover, the special linear group admits the Iwasawa decomposition
\[
SL_2(\mathbb{R}) = K\times A\times N.
\]
The quotient $GL_2(\mathbb{R})/CO_2(\mathbb{R})$ is then the direct product $A\times N$.
\end{example}

\end{document}